\documentclass[11pt]{amsart}

\usepackage{palatino, mathpazo}
\usepackage{amsfonts}
\usepackage{amsmath}
\usepackage{amssymb,latexsym,xcolor}
\usepackage{graphicx}
\usepackage{harpoon}
\usepackage[mathscr]{eucal}
\usepackage{amssymb}%
\usepackage[
linktocpage=true,colorlinks,citecolor=magenta,linkcolor=blue,urlcolor=magenta]{hyperref}
\providecommand{\U}[1]{\protect \rule{.1in}{.1in}}

\newtheorem{theorem}{Theorem}[section]

\newtheorem{lemma}[theorem]{Lemma}

\theoremstyle{remark}
\newtheorem{remark}[theorem]{Remark}

\numberwithin{equation}{section}
\begin{document}
\title[the weight six Eisenstein series]{On the distribution of critical points of the Eisenstein series $E_6$ and monodromy interpretation}
\author{Zhijie Chen}
\address{Department of Mathematical Sciences, Yau Mathematical Sciences Center,
Tsinghua University, Beijing, 100084, China }
\email{zjchen2016@tsinghua.edu.cn}

\begin{abstract}
In previous works joint with Lin, we proved that the Eisenstein series $E_4$ (resp. $E_2$) has at most one critical point in every fundamental domain $\gamma(F_0)$ of $\Gamma_{0}(2)$, where $\gamma(F_0)$ are translates of the basic fundamental domain $F_0$ via the M\"{o}bius transformation of $\gamma\in\Gamma_{0}(2)$. But the method can not work for the Eisenstein series $E_6$.

In this paper, we develop a new approach to show that $E_6'(\tau)$ has exactly either $1$ or $2$ zeros in every fundamental domain $\gamma(F_0)$ of $\Gamma_{0}(2)$. A criterion for $\gamma(F_0)$ containing exactly $2$ zeros is also given.  Furthermore, by mapping all zeros of $E_6'(\tau)$ into $F_0$ via the M\"{o}bius transformations of $\Gamma_{0}(2)$ action, the images give rise to a dense subset of the union of three disjoint smooth curves in $F_0$. A monodromy interpretation of these curves from a complex linear ODE is also given. As a consequence, we give a complete description of the distribution of the zeros of $E_6'(\tau)$ in fundamental domains of $SL(2,\mathbb{Z})$.
\end{abstract}
\maketitle

\section{Introduction}

As the first nontrivial examples of modular forms, the Eisenstein series have always played fundamental
roles in the theory of modular forms and elliptic curves since their discovery in the early 19th century. On the other hand,
besides their numerous applications, the Eisenstein series are rather deep objects by themselves (cf. \cite{KY}).

This paper is the final one in our project of understanding the critical points of the basic Eisenstein series $E_{2k}(\tau)$, $k\leq 3$. Joint with Lin, we studied the distribution of critical points of $E_2(\tau)$ and $E_4(\tau)$ in \cite{CL-E2,CL-E4}, which already had interesting applications in studying the deformation of the spectrum of the Lam\'{e} operator $\frac{d^2}{dx^2}-6\wp(x;\tau)$ as $\tau$ varies (see \cite{CFL-AIM}). However, the method in \cite{CL-E2,CL-E4} does not work for $E_{6}(\tau)$. The aim of this paper is to give a complete description of the critical points of $E_{6}(\tau)$, or equivalently the well-known invariant $g_3(\tau)$ in the theory of elliptic curves, by developing \emph{a different approach} from \cite{CL-E2,CL-E4}.

Denote $\Lambda_{\tau}=\mathbb{Z+Z}\tau$,
where $\tau \in \mathbb{H}=\{ \tau|\operatorname{Im}\tau>0\}$.
Let $\wp(z)=\wp(z;\tau)$ be the Weierstrass  $\wp$-function with periods $\Lambda_{\tau}$, defined by
\[
\wp(z;\tau):=\frac{1}{z^{2}}+\sum_{\omega \in \Lambda_{\tau}\backslash
\{0\}}\left(  \frac{1}{(z-\omega)^{2}}-\frac{1}{\omega^{2}}\right).
\]
It is well known that $\wp(z;\tau)$ satisfies the following cubic equation:
\[\wp'(z;\tau)^2=4\wp(z;\tau)^3-g_2(\tau)\wp(z;\tau)-g_3(\tau).\]
Here $g_2(\tau), g_3(\tau)$ are invariants of the elliptic curve $E_{\tau}:=\mathbb{C}/\Lambda_{\tau}$, which are just multiples of the basic two Eisenstein series $G_4(\tau), G_6(\tau)$, or equivalently the normalized $E_4(\tau), E_6(\tau)$, respectively: \[g_2(\tau)=\frac{4\pi^4}{3}E_4(\tau)=60 G_4(\tau):=60\mathop{{\sum}'}_{(m,n)\in\mathbb{Z}^2}\frac{1}{(m\tau+n)^4},\]
\[ g_3(\tau)=\frac{8\pi^6}{27}E_6(\tau)=140 G_6(\tau):=140\mathop{{\sum}'}_{(m,n)\in\mathbb{Z}^2}\frac{1}{(m\tau+n)^6},\]
where $\mathop{{\sum}'}$ means to sum over $(n,m)\in \mathbb{Z}^2\setminus\{(0,0)\}$ conventionally. The derivative of $E_6(\tau)$ is given by Ramanujan's formula (cf. \cite{Ramanujan})
\begin{equation}\label{Rama-for}E_6'(\tau)=\pi i(E_2(\tau)E_6(\tau)-E_4(\tau)^2),\end{equation}
where $E_{2}(\tau)$ is the Eisenstein series of weight $2$ defined by
\[E_2(\tau)
:=\frac{3}{\pi^2}\sum_{m=-\infty}^{\infty}\mathop{{\sum}'}_{n=-\infty}^{\infty}\frac{1}{(m\tau+n)^2}.\]

It is well known that $E_4(\tau), E_6(\tau)$ are both modular forms for $SL(2,\mathbb{Z})$ of weight $4, 6$ respectively, while $E_2(\tau)$ is not a modular form but only a \emph{quasimodular form}, and so is $E_6'(\tau)$. This indicates that the set of zeros of $E_6'(\tau)$ is \emph{not stable} under the modular
group $SL(2,\mathbb{Z})$ and hence difficult to study. There are many recent works studying the zeros of quasimodular forms including critical points of modular forms; see e.g. \cite{BG,EIBS,Kaneko,SS,WY} and references therein. In particular, Saber and Sebbar \cite{SS} proved that \emph{for each modular form $f$ for a subgroup of $SL(2,\mathbb{Z})$, its derivative $f'$ has infinitely many inequivalent zeros and all, but a finite number, are simple}. As an example, they showed that $E_6'(\tau)$ has infinitely many inequivalent zeros which are all simple. However, the distribution of the zeros of $E_6'(\tau)$ is still far from being understood.

Due to the fundamental importance of $E_{6}(\tau)$ in the theory of modular forms (For example, $E_2, E_4$ and $E_6$ play crucial roles in Viazovska's proof of the sphere packing problem \cite{Viazovska}),
in this paper we develop our own approach to give a complete description of its critical points. To compare with $E_2(\tau)$ and $E_4(\tau)$,
we would like to state our results in fundamental domains of $\Gamma_0(2)$ first.

\subsection{Distribution in fundamental domains of $\Gamma_0(2)$}
Recall that $\Gamma_{0}(2)$ is the congruence subgroup of $SL(2,\mathbb{Z})$ defined by
\[
\Gamma_{0}(2):=\left \{  \left.
\bigl(\begin{smallmatrix}a & b\\
c & d\end{smallmatrix}\bigr)
\in SL(2,\mathbb{Z})\right \vert c\equiv0\text{ }\operatorname{mod}2\right \},
\]
and $F_{0}$ is the basic fundamental domain of $\Gamma
_{0}(2)$ given by
\[
F_{0}:=\{ \tau \in \mathbb{H}\ |\ 0\leqslant  \text{Re}  \tau \leqslant
1\  \text{and}\ |\tau-\tfrac{1}{2}|\geqslant \tfrac{1}{2}\}.
\]
Then for any $\gamma=%
\bigl(\begin{smallmatrix}a & b\\
c & d\end{smallmatrix}\bigr)
\in \Gamma_{0}(2)/\{ \pm I_{2}\}$ (i.e. {\it consider $\gamma$ and $-\gamma$ to be
the same}),%
\begin{equation}\label{gammaF0}
\gamma (F_{0}):=\left \{  \left.  \gamma \cdot \tau:=\tfrac{a\tau+b}{c\tau
+d}\right \vert \tau \in F_{0}\right \}  =(-\gamma)(F_{0})%
\end{equation}
is also a fundamental domain of $\Gamma_{0}(2)$. Note that $\gamma ( F_{0})=F_{0}+m$ for some $m\in \mathbb{Z}$ if and only if $c=0$.

Our first result shows that the critical points of $E_6(\tau)$ satisfy the following distribution.

\begin{theorem}\label{thm-number}
Let $\gamma (F_{0})$ be a fundamental domain of $\Gamma_{0}(2)$ with $\gamma=
\bigl(\begin{smallmatrix}a & b\\
c & d\end{smallmatrix}\bigr)
\in \Gamma_{0}(2)/\{ \pm I_{2}\}$. Then the following statements hold.
\begin{itemize}
\item[(1)]$E_6(\tau)$ has exactly one critical point in $\gamma(F_0)$ if $c=0$.
\item[(2)] $E_6(\tau)$ has exactly two different critical points in $\gamma(F_0)$ if $c\neq 0$.
\end{itemize}
\end{theorem}

\begin{remark}\label{rmk-ss} Saber and Sebbar \cite[p 1787]{SS} said that on the line $\operatorname{Re} \tau=1/2$, $E_6'(\tau)$ is purely imaginary and has a unique zero $\tau_\infty=1/2+i0.6341269863$. Noting $\tau_\infty\in F_0$, our Theorem \ref{thm-number}-(1) improves this result by showing that this $\tau_\infty$ is indeed the unique zero of $E_6'(\tau)$ in $F_0$, i.e. $E_6'(\tau)\neq 0$ in $F_0\setminus\{\tau_\infty\}$.
\end{remark}

A refined version of Theorem \ref{thm-number} will be given in Theorem \ref{thm-number-1}, where we will see that the critical points of $E_6(\tau)$ actually lie in $\gamma(\mathring{F}_{0})$, where $\mathring{F}_{0}=F_{0}\backslash \partial F_{0}$ denotes the
set of interior points of $F_{0}$.
It is interesting to compare Theorem \ref{thm-number} with our previous results concerning $E_4(\tau)$ and $E_2(\tau)$.

\medskip
\noindent{\bf Theorem A.} \cite{CL-E2,CL-E4} {\it
Under the same notations in Theorem \ref{thm-number}, the following statements hold.
\begin{itemize}
\item[(1)] $E_{4}(\tau)$ (resp. $E_2(\tau)$) has no critical points in $\gamma (F_{0})$ if $c=0$.
\item[(2)]
$E_4(\tau)$ (resp. $E_2(\tau)$) has exactly one critical point in $\gamma(F_0)$ if $c\neq 0$.
\end{itemize}
}

\begin{remark}
Theorem \ref{thm-number} shows that the situation of $E_6(\tau)$ is different from those of $E_2(\tau)$ and $E_4(\tau)$. Furthermore, the main ideas of the proofs are also different. The proof for $E_4(\tau)$ (resp. $E_2(\tau)$) relies essentially on an auxiliary pre-modular form of weight $6$ (resp. of weight $3$) introduced in \cite{LW2} (See \cite{LW2} for the definition of pre-modular forms). In particular,
\begin{equation} \label{eq: data}
  \parbox{\dimexpr\linewidth-4em}{
  Theorem A-(1) was obtained as a consequence of Theorem A-(2).
  }
\end{equation}
 Unfortunately, we still do not know what the auxiliary pre-modular form related to $E_6(\tau)$ is, and so the above idea in \cite{CL-E2,CL-E4} does not work. In this paper, we develop a different idea for $E_6(\tau)$. In contrast with \eqref{eq: data}, we apply the method of continuity to prove Theorem \ref{thm-number}-(1) first and then prove Theorem \ref{thm-number}-(2) as a consequence of (1). 
\end{remark}

In view of Theorem \ref{thm-number}, we can transform every critical point of $E_{6}(\tau)$ via the M\"{o}bius transformation of $\Gamma_{0}(2)$ action to locate it in $F_0$. Denote \emph{the collection of such corresponding points in $F_0$ by $\mathcal{D}_0$}, i.e. $\mathcal{D}_0$ is the set of those points in $F_0$ which are obtained by
transforming the critical points of $E_{6}(\tau)$ via the M\"{o}bius transformations of $\Gamma_{0}(2)$ action. A fundamental question related to the distribution of the critical points is: {\it What is the geometry of the set $\mathcal{D}_0$}? We answer this question in our second result.

\begin{theorem}\label{Location} There exist three disjoint smooth curves $\mathcal{C}_1, \mathcal{C}_2, \mathcal{C}_3$ in $F_0$ satisfying
\[\partial \mathcal{C}_{1}=\{0,1\},\;\;\partial \mathcal{C}_{2}=\{0,\tfrac{1}{4}+i\infty\}, \;\;\partial \mathcal{C}_{3}=\{1,\tfrac{3}{4}+i\infty\}\]
such that
\begin{equation}\label{cf0}\mathcal{D}_0\subset \mathcal{C}_{1}\cup\mathcal{C}_{2}\cup\mathcal{C}_{3}=\overline{\mathcal{D}_0}\cap F_0= \overline{\mathcal{D}_0}\setminus\{0,1,\infty\}.\end{equation}
\end{theorem}

\begin{remark} Theorem \ref{Location} shows that $\mathcal{D}_0$ is \emph{a dense subset} of the union of the three curves, the boundaries of which are precisely the cusps $\{0,1,\infty\}$ of $F_0$. We will see in Section \ref{mono-int} that the three curves have interesting monodromy meanings from a complex linear ODE.
We will also prove in Section \ref{smoothcurves} that $\mathcal{C}_1$ can be parameterized by $C\in(1,+\infty)\cup \{\infty\}\cup(-\infty, 0)$ (resp. $\mathcal{C}_2$ is parameterized by $C\in(0, +\infty)$ and $\mathcal{C}_3$ by $C\in(-\infty, 1)$) via the following identity
\begin{equation}\label{c=tau}C=\tau+\frac{36\pi i g_3(\tau)}{(g_{2}^2-18\eta_1g_3)(\tau)},\quad \tau\in \mathcal{C}_1\cup\mathcal{C}_{2}\cup\mathcal{C}_{3}.\end{equation}
Here $\eta_1(\tau):=2\zeta(\frac{1}{2};\tau)$ is a quasi-period of the Weierstrass zeta
function $\zeta(z)=\zeta(z;\tau):=-\int^{z}\wp(\xi;\tau)d\xi$. Indeed it is a multiple of the Eisenstein series $E_{2}(\tau)$: $\eta_1(\tau)=\frac{\pi^2}{3}E_2(\tau)$. The relation between (\ref{c=tau}) and $E_6'(\tau)$ comes from Ramanujan's formula (\ref{Rama-for}), which can be written in terms of $g_2, g_3, \eta_1$ as follows (see also \cite[p.704]{YB}):
\begin{equation}\label{dd-eta11}
g_{3}^{\prime}(\tau)=\tfrac{-i}{6\pi}\left(g_2(\tau)^2-18\eta_1(\tau)g_{3}(\tau)\right).
\end{equation}
It is interesting to compare Theorem \ref{Location} with that of $E_{4}(\tau)$ proved in \cite{CL-E4}: Under the M\"{o}bius transformations of $\Gamma_{0}(2)$ action, the images of all critical points of $E_{4}(\tau)$ in $F_0$ form a dense subset of the union of three disjoint smooth curves $\mathcal{C}_-, \mathcal{C}_0, \mathcal{C}_+$ in $F_0$. The different point is that $\mathcal{C}_-, \mathcal{C}_0, \mathcal{C}_+$ are parameterized by $C\in (-\infty, 0), C\in (0, 1)$ and $C\in (1,+\infty)$, respectively, via the following identity
\[
C=\tau-\frac{4\pi i g_2(\tau)}{(2\eta_{1}g_{2}-3g_3)(\tau)},\quad \tau\in \mathcal{C}_-\cup\mathcal{C}_0\cup\mathcal{C}_+,
\]
i.e. the ranges of the parameter $C$ are disjoint for any two of $\{\mathcal{C}_-,\mathcal{C}_0, \mathcal{C}_+\}$ for $E_4$; but the ranges of the parameter $C$ have intersections for any two of $\{\mathcal{C}_{1}, \mathcal{C}_{2}, \mathcal{C}_{3}\}$ for $E_6$.
\end{remark}

\subsection{Distribution in fundamental domains of $SL(2,\mathbb{Z})$}
In view of Theorem \ref{thm-number}, it is natural for us to study further the distribution of critical points in fundamental domains of the modular group $SL(2,\mathbb{Z})$ for two reasons: (i) $E_6(\tau)$ is a modular form for $SL(2,\mathbb{Z})$; (ii) A fundamental domain of $SL(2,\mathbb{Z})$ is smaller than that of $\Gamma_0(2)$, so we wonder whether the two zeros of $E_6'(\tau)$ in $\gamma(F_0)$ with $c\neq0$ can be separated or not by different fundamental domains of $SL(2,\mathbb{Z})$.
It is convenient for us to choose
\begin{equation}\label{F-fun}
F:=\{ \tau \in \mathbb{H}\ |\ 0\leqslant \operatorname{Re}\tau\leqslant 1,|\tau|\geqslant
1,|\tau-1|\geqslant1\}
\end{equation}
to be the basic fundamental domain of $SL(2,\mathbb{Z})$, because $F\subset F_0$
and (see Figure 1, which is copied from \cite{CKLW})
\begin{equation}
F_{0}=F\cup \gamma_{1}(F)\cup \gamma_{2}(F),\quad \gamma_{1}:=\bigl(\begin{smallmatrix}0 & 1\\
-1 & 1\end{smallmatrix}\bigr)
,\  \  \gamma_{2}:=\bigl(\begin{smallmatrix}1 & -1\\
1 & 0\end{smallmatrix}\bigr). \label{F-de}%
\end{equation}

\begin{figure}[btp]\label{fundamental-dom}
\includegraphics[width=3.6in]{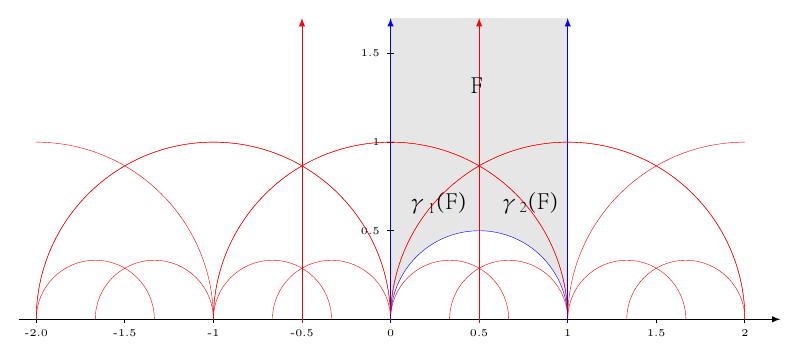}\caption{$F_{0}=F\cup
\gamma_{1}(F)\cup\gamma_{2}(F)$.}%
\end{figure}

Our following result shows that $E_6(\tau)$ has at most one critical point in $\gamma(F)$ for any $\gamma\in SL(2,\mathbb{Z})$.
\begin{theorem}\label{thm-number-F}
Let $\gamma (F)$ be a fundamental domain of $SL(2,\mathbb{Z})$ with $\gamma=
\bigl(\begin{smallmatrix}a & b\\
c & d\end{smallmatrix}\bigr)
\in SL(2,\mathbb{Z})/\{ \pm I_{2}\}$. Then the following statements hold.
\begin{itemize}
\item[(1)]$E_6(\tau)$ has no critical points in $\gamma(F)$ if $\frac{-d}{c}\in (0,1)\cup\{\infty\}$.
\item[(2)] $E_6(\tau)$ has exactly one critical point in $\gamma(F)$ if $\frac{-d}{c}\in (-\infty, 0]\cup [1,+\infty)$. Moreover, the unique critical point lies on the boundary $\partial\gamma(F)=\gamma(\partial F)$ if and only if $\frac{-d}{c}\in \{0,1\}$.
\end{itemize}
\end{theorem}

\begin{remark}
The equivalence of Theorems \ref{thm-number} and \ref{thm-number-F} is not obvious at this stage. For example, Theorem \ref{thm-number-F} shows that $E_6(\tau)$ has no critical points in $F$ and has a unique critical point $\tilde{\tau}_j$ in $\gamma_j(F)$, $j=1,2$, while Theorem \ref{thm-number} shows that $E_6(\tau)$ a unique critical point in $F_0$. To coincide these two assertions, we must have  $\tilde{\tau}_1=\tilde{\tau}_2\in\partial\gamma_1(F)\cap \partial\gamma_2(F)$, i.e. $\tilde{\tau}_1=\frac{1}{2}+ib$ with $b\in [\frac{1}{2},\frac{\sqrt{3}}{2}]$. Indeed, this $\tilde{\tau}_1$ is precisely the $\tau_\infty$ given in Remark \ref{rmk-ss} and also Theorem \ref{uniquezero} below.
\end{remark}

Similarly, we can transform every critical point of $E_{6}(\tau)$ via the M\"{o}bius transformation of $SL(2,\mathbb{Z})$ action to locate it in $F$. Denote \emph{the collection of such corresponding points in $F$ by $\mathcal{D}$}. Recalling the smooth curves $\mathcal{C}_k$'s in Theorem \ref{Location}, we
define two disjoint sets
\begin{equation}\label{C-F}
\mathcal{C}_{<}:=\mathcal{C}_2\cap F,\quad \mathcal{C}_{>}:=\mathcal{C}_3\cap F.
\end{equation}
In general, $\mathcal{C}_{<}$ or $\mathcal{C}_{>}$ might consist of several disjoint curves. Our following result shows that this possibility does not happen.
\begin{theorem}\label{Location-F} Under the above notations, the following statements hold.
\begin{itemize}
\item[(1)] Each of $\mathcal{C}_{<}\setminus \partial \mathcal{C}_{<}$ and $\mathcal{C}_{>}\setminus \partial \mathcal{C}_{>}$ is a smooth curve in $F$ with
\[\partial \mathcal{C}_{<}=\{\tilde{\tau}_<,\tfrac{1}{4}+i\infty\}, \;\;\partial \mathcal{C}_{>}=\{\tilde{\tau}_>,\tfrac{3}{4}+i\infty\}\]
for some $\tilde{\tau}_<$, $\tilde{\tau}_>\in \partial F$.
\item[(2)] $\mathcal{D}$ is a dense subset of the union of the two disjoint curves $\mathcal{C}_{<}$ and $\mathcal{C}_{>}$, i.e.
\begin{equation}\label{cf0-F}\mathcal{D}\subset \mathcal{C}_{<}\cup\mathcal{C}_{>}=\overline{\mathcal{D}}\cap F= \overline{\mathcal{D}}\setminus\{\infty\}.\end{equation}
\end{itemize}
\end{theorem}

Again the equivalence of Theorems \ref{Location} and \ref{Location-F} is not obvious at this stage.
The above results completely describe the distribution of all critical points of $E_{6}(\tau)$ or equivalently $g_3(\tau)$.

The rest of this paper is organized as follows. In Section \ref{uniqueness-F0}, we apply the method of continuity to prove Theorem \ref{thm-number}-(1). As a consequence, Theorem \ref{thm-number}-(2) will be proved in Section \ref{twocriticalp}. In Section \ref{smoothcurves}, we introduce the precise definition of the three curves $\mathcal{C}_1,\mathcal{C}_2, \mathcal{C}_3$ and prove Theorem \ref{Location}. We also restate Theorem \ref{thm-number} more accurately in Theorem \ref{thm-number-1}.
In Section \ref{distribution-F}, we prove Theorems \ref{thm-number-F} and \ref{Location-F} as consequences of Theorems \ref{thm-number} and \ref{Location}.
In Section \ref{mono-int}, we introduce a monodromy interpretation from a complex linear ODE for the curves.

After this paper was submitted, recently we learned from Professor Bonk that Gun-Oesterl\'{e} \cite{GO22} studied the critical points of Eisenstein series $E_{2k}$ for $k\geq 2$ in a unified way, and Bonk \cite{Bonk} studied the critical points of Eisenstein series $E_{2k}$ for $k\in\{1,2,3\}$ from the viewpoint of conformal maps. The methods in \cite{Bonk, GO22} are different from ours.

\section{Uniqueness of the critical point in $F_0$}
\label{uniqueness-F0}

The purpose of this section is to prove Theorem \ref{thm-number}-(1). Remark that $c=0$ implies $\gamma(F_0)=F_0+m$ for some $m\in\mathbb{Z}$.

\begin{theorem}[=Theorem \ref{thm-number}-(1)]\label{uniquezero}
The derivative $g_3'(\tau)$ has a unique zero in $F_0$, which is denoted by $\tau_{\infty}$. Furthermore, $\tau_{\infty}=\frac{1}{2}+i b_{\infty}$ with $b_{\infty}\in (\frac{1}{2}, \frac{\sqrt{3}}{2})$.

Consequently, $\tau_{\infty}+m$ is the unique zero of $g_3'(\tau)$ in $F_0+m$ for any $m\in\mathbb{Z}$, i.e. Theorem \ref{thm-number}-(1) holds.
\end{theorem}

\begin{remark}
Clearly the last assertion of Theorem \ref{uniquezero} follows from $g_3(\tau)=g_3(\tau+1)$.
 Remark \ref{rmk-ss} indicates that $b_\infty=0.6341269863$ by applying Saber and Sebbar's result \cite[p 1787]{SS}.
\end{remark}
First we recall the modularity of $g_2(\tau), g_3(\tau)$, $\wp(z;\tau)$ and $\zeta(z;\tau)$. Given any $\bigl(\begin{smallmatrix}a & b\\
c & d\end{smallmatrix}\bigr)
\in SL(2,\mathbb{Z})$, it is well known that (cf. \cite{A,Akhiezer})
\begin{equation}g_{2}(\tfrac{a\tau+b}{c\tau+d})=(c\tau+d)^4g_{2}(\tau),\quad
g_{3}(\tfrac{a\tau+b}{c\tau+d})=(c\tau+d)^6g_{3}(\tau),\label{II-30-00}\end{equation}
\[
\wp \Big(\tfrac{z}{c\tau+d}; \tfrac{a\tau+b}{c\tau+d}\Big)  =(
c\tau+d)^{2}\wp (z;\tau),\]
\[\zeta \Big(\tfrac{z}{c\tau+d}; \tfrac{a\tau+b}{c\tau+d}\Big)
=(c\tau+d)  \zeta( z;\tau).
\]
Also recall that $\eta_{1}(\tau):=2\zeta(\frac{1}{2};\tau)=\frac{\pi^2}{3}E_2(\tau)$ and $\eta_{2}(\tau):=2\zeta(\frac{\tau}{2};\tau)$ are the two quasi-periods of $\zeta(z;\tau)$:
\begin{equation}
\eta_{1}(\tau)=\zeta(z+1;\tau)-\zeta(z;\tau),\text{ \ }\eta_{2}(\tau
)=\zeta(z+\tau;\tau)-\zeta(z;\tau). \label{quasi}%
\end{equation}
The well-known Legendre relation gives $\eta_2(\tau)=\tau\eta_{1}(\tau)-2\pi i$.
It is easy to see that
\begin{equation}%
\begin{pmatrix}
\eta_{2}(\tfrac{a\tau+b}{c\tau+d})\\
\eta_{1}(\tfrac{a\tau+b}{c\tau+d})
\end{pmatrix}
=(c\tau+d)\begin{pmatrix}
a & b\\
c & d
\end{pmatrix}
\begin{pmatrix}
\eta_{2}(\tau)\\
\eta_{1}(\tau)
\end{pmatrix}
. \label{II-31-1-00}%
\end{equation}
Thus each $\eta_j(\tau)$ is not a modular form, but $(\eta_2(\tau),\eta_1(\tau))$ is the so-called \emph{vector-valued modular form}.
In the rest of this paper, we will freely use the formulas (\ref{II-30-00})-(\ref{II-31-1-00}).

Recall (\ref{dd-eta11}) that
\begin{equation}\label{d-eta1}g_3'(\tau)=\tfrac{-i}{6\pi}(g_2^2-18\eta_1 g_3)(\tau).\end{equation}
Thus, it suffices for us to study the function $(g_2^2-18\eta_1 g_3)(\tau)$.

\begin{lemma}\label{lemma-simplezero}
Any zero of $(g_2^2-18\eta_1 g_3)(\tau)$ is simple.
\end{lemma}

\begin{proof}
This lemma is equivalent to assert that any zero of $g_3'(\tau)$ is simple, which was already pointed out in \cite[p.1787]{SS}. Here we sketch the proof for later usage.
Applying the classical formulas (see e.g. \cite[p.704]{YB} or \cite[p.1786]{SS})
\[\eta_1'(\tau)=\tfrac{i}{24\pi}(12\eta_1^2-g_2)(\tau),\]
\[g_2'(\tau)=\tfrac{i}{\pi}(2\eta_1g_2-3g_3)(\tau),\]
a direct computation gives
\begin{equation}\label{2derivative}(g_2^2-18\eta_1 g_3)'(\tau)=\frac{7i}{4\pi}(4\eta_1 g_2^2-36\eta_1^2g_3-3g_2g_3)(\tau).\end{equation}
Suppose $g_2^2-18\eta_1 g_3=0$. Inserting this into (\ref{2derivative}) leads to
\begin{equation}\label{2derivative-1}(g_2^2-18\eta_1 g_3)'=\frac{21i}{4\pi}g_3(12\eta_1^2-g_2).\end{equation}
If $(g_2^2-18\eta_1 g_3)'=0$, then $12\eta_1^2-g_2=0$. This, together with $g_2^2-18\eta_1 g_3=0$, yields $g_2^3-27g_3^2=0$, a contradiction. Thus $(g_2^2-18\eta_1 g_3)'\neq 0$ as long as  $g_2^2-18\eta_1 g_3=0$.
\end{proof}

We need to use the method of continuity to prove Theorem \ref{uniquezero}.
For $t\in [0,1]$, we define holomorphic functions
\begin{equation}\label{t-deform-1}
h_{t}(\tau):=g_2(\tau)^2-18t\eta_1(\tau) g_3(\tau).
\end{equation}

\begin{lemma}\label{lemma-etai}
If $\tau\in i\mathbb{R}_{>0}$, then $g_{3}'(\tau)\neq 0$ and so
\begin{equation}\label{g3eta1-0}
h_t(\tau)=g_2(\tau)^2-18t\eta_1(\tau) g_3(\tau)>0,\quad \forall t\in [0,1].
\end{equation}
\end{lemma}

\begin{proof}
Denote $q=e^{2\pi i\tau}$ and recall the $q$-expansion of $g_{3}(\tau)$ (cf. \cite[p.44]{Lang}):
\begin{align}
g_{3}(\tau)  =\frac{8\pi^{6}}{27}\bigg(1-504\sum_{k=1}^{\infty}\sigma
_{5}(k)q^{k}\bigg),\;\,\text{where }\sigma_{5}(k)=\sum_{1\leq d|k}d^{5}.\label{ex-2}
\end{align}
Let $\tau=ib$ with $b>0$. Then $q=e^{-2\pi b}$ and hence $\frac{d}{db}g_{3}(ib)>0$ for $b>0$. So $g_{3}'(\tau)\neq 0$ and we see from (\ref{d-eta1}) that
\begin{equation}\label{g3eta1}
g_2(\tau)^2-18\eta_1(\tau) g_3(\tau)>0.
\end{equation} Recalling the well known fact (cf. \cite{Lang}) that
\begin{equation}\label{ib-real}g_2(\tau)>0, \quad \eta_1(\tau), g_3(\tau)\in\mathbb{R}\quad \text{for }\tau\in i\mathbb{R}_{>0},\end{equation}
we obtain (\ref{g3eta1-0}).
\end{proof}

\begin{lemma}
\label{lemma-t10}For any $t\in [0,1]$, $h_{t}(\tau
)\not =0$ for $\tau \in \partial F_{0}\cap \mathbb{H}$.
\end{lemma}

\begin{proof}
Suppose $h_{t}(\tau)=0$ for some $t\in [0,1]$ and $\tau \in \partial F_{0}\cap \mathbb{H}$. By (\ref{g3eta1-0}) we have $\tau \not\in i\mathbb{R}_{>0}$. If $\tau\in i\mathbb{R}_{>0}+1$, then it follows from $g_k(\tau-1)=g_k(\tau)$, $k=2,3$, and $\eta_1(\tau-1)=\eta_1(\tau)$ that
\[h_t(\tau-1)=g_2(\tau-1)^2-18t\eta_1(\tau-1) g_3(\tau-1)=h_{t}(\tau)=0,\]
a contradiction with (\ref{g3eta1-0}) and $\tau-1\in i\mathbb{R}_{>0}$.

Therefore, $|\tau-\frac{1}{2}|=\frac{1}{2}$ and then $\tau^{\prime}:=\frac{\tau}{1-\tau}\in i\mathbb{R}_{>0}$. By $g_{2}(\tau^{\prime})=(1-\tau)^{4}g_{2}(\tau)$, $g_{3}(\tau^{\prime})=(1-\tau)^{6}g_{3}(\tau)$
and%
\begin{equation}
\eta_{2}(\tau^{\prime})=(1-\tau)\eta_{2}(\tau),\text{ }\eta_{1}(\tau^{\prime
})=(1-\tau)(\eta_{1}(\tau)-\eta_{2}(\tau)), \label{c-5}%
\end{equation}
we easily obtain
\begin{align*}
&(-1-\tau')g_2(\tau')^2-18t(-\eta_1(\tau')-\eta_2(\tau'))g_3(\tau')\\
=&-(1-\tau)^7(g_2(\tau)^2-18t\eta_1(\tau)g_3(\tau))=(\tau-1)^7h_{t}(\tau)=0.
\end{align*}
This, together with $\eta_2(\tau')=\tau'\eta_1(\tau')-2\pi i$, $\tau'\in i\mathbb{R}_{>0}$ and (\ref{g3eta1-0}), implies
\[-1=\tau'+i\frac{36 t\pi g_3(\tau')}{(g_2^2-18t\eta_1 g_3)(\tau')}\in i\mathbb{R},\]
a contradiction. The proof is complete.
\end{proof}

\begin{lemma}
\label{infinity-behavior-t}Let $t\in [0,1)$ and $\tau\in F_0$. Then
\begin{equation}\label{infty-be-t}
\lim_{\tau\to\infty}h_t(\tau)=\tfrac{16}{9}(1-t)\pi^8>0;
\end{equation}
\begin{equation}\label{0-be-t}
\lim_{\tau\to 0}h_t(\tau)= \infty;\quad  \lim_{\tau\to 1}h_t(\tau) =\infty.
\end{equation}
\end{lemma}

\begin{proof} Denote $q=e^{2\pi i \tau}$ and recall the well known $q$-expansions (cf. \cite[p.44]{Lang})
\begin{equation}\label{q-exp-eta1}
\eta_{1}(\tau)=\tfrac{1}{3}\pi^{2}(1-24q-72q^{2})+O(|q|^{3}),
\end{equation}
\[
g_{2}(\tau)=\tfrac{4}{3}\pi^{4}(1+240(q+9q^{2}))+O(|q|^{3}), \]
\begin{equation}\label{q-exp-g3}
g_{3}(\tau)=\tfrac{8}{27}\pi^{6}(1-504(q+33q^{2}))+O(|q|^{3}). \end{equation}
Clearly (\ref{infty-be-t}) follows from (\ref{t-deform-1}) and these $q$-expansions (\ref{q-exp-eta1})-(\ref{q-exp-g3}).  For $t\in [0,1)$ and $C\in\mathbb{R}$, we define
\begin{align*}
f_{t,C}(\tau):=&(C-\tau)g_2(\tau)^2-18t(C\eta_1(\tau)-\eta_2(\tau))g_3(\tau)\\
=&(C-\tau)h_t(\tau)-36t\pi i g_3(\tau).
\end{align*}
Then
\begin{equation}
f_{t,C}(\tau)=-\tfrac{16}{9}(1-t)\pi^8\tau+O(1)\to \infty\quad\text{as $F_0\ni \tau\to \infty$}.
\end{equation}

To study the asymptotics at the cusp $0$,
we let $\hat{\tau}:=\frac{\tau-1}{\tau}$.
 By $g_{2}(\hat{\tau})=\tau^{4}g_{2}(\tau)$, $g_{3}(\hat{\tau})=\tau^{6}g_{3}(\tau)$
and%
\begin{equation}
\eta_{1}(\hat{\tau})=\tau\eta_{2}(\tau),\text{ }\eta_{2}(\hat{\tau})=\tau(\eta_{2}(\tau)-\eta_{1}(\tau)), \label{c-c-5}%
\end{equation}
a direct computation gives $f_{t, 1}(\hat{\tau})=\tau^7 h_{t}(\tau)$. When $\tau \in F_{0}$ and $\tau \rightarrow0$, we have $\hat{\tau}=\frac{\tau-1}{\tau}\in
F_{0}$ and $\hat{\tau}\rightarrow \infty$. Consequently,
\begin{align}
h_{t}(\tau)=\frac{ f_{t,1}(\hat{\tau})}{\tau^{7}}=\frac{1}{\tau^{7}}(-\tfrac{16(\tau-1)}{9\tau}(1-t)\pi^8+O(1))\quad\text{as $\tau\to 0$}.\label{ayp-8-t}%
\end{align}
This proves the first formula in (\ref{0-be-t}).

Finally we let $\hat{\tau}:=\frac{1}{1-\tau}$. By $g_{2}(\hat{\tau})=(1-\tau)^{4}g_{2}(\tau)$, $g_{3}(\hat{\tau})=(1-\tau)^{6}g_{3}(\tau)$
and%
\begin{equation}
\eta_{1}(\hat{\tau})=(1-\tau)(\eta_1(\tau)-\eta_{2}(\tau)),\text{ }\eta_{2}(\hat{\tau})=(1-\tau)\eta_{1}(\tau), \label{cc-c-5}%
\end{equation} we easily obtain $f_{t, 0}(\hat{\tau})=(\tau-1)^7 h_{t}(\tau)$. When $\tau \in F_{0}$ and $\tau \rightarrow1$, we have $\hat{\tau}=\frac{1}{1-\tau}\in
F_{0}$ and $\hat{\tau}\rightarrow \infty$.
Consequently,
\begin{align}\label{eq-h-t}
h_{t}(\tau)=\frac{f_{t,0}(\hat{\tau})}{(\tau-1)^{7}}
=\frac{1}{(\tau-1)^{7}}(\tfrac{-16}{9(1-\tau)}(1-t)\pi^8+O(1))\;\text{as $\tau\to 1$},
\end{align}
which implies the second formula in (\ref{0-be-t}).
\end{proof}

\begin{lemma}
\label{lemma-ht}For any $t\in (0,1)$, $h_{t}(\tau
)$ has exactly two different zeros in $F_{0}$, which both lie on the line $\operatorname{Re}\tau=\frac{1}{2}$ and are simple zeros.
\end{lemma}

\begin{proof}
For $t\in [0,1)$, Lemma \ref{infinity-behavior-t} implies
\[
h_t(\tau)\not \rightarrow 0\;\text{ as }\;F_{0}\ni \tau \rightarrow
\infty,0,1\;\text{respectively}.
\]
Together with Lemma \ref{lemma-t10} that $h_t
(\tau)\not =0$ on $\partial F_{0}\cap \mathbb{H}$, it is easy to apply the
argument principle to conclude that \emph{the number of zeros (up to multiplicity) of $h_t
(\tau)$ in $F_{0}$ is a constant for $t\in [0,1)$}.
Since it is well known that $h_0(\tau)=g_2(\tau)^2$ has a unique zero $\tau=e^{\pi i/3}$ of multiplicity $2$ in $F_0$, we conclude that $h_{t}(\tau
)$ has at most two different zeros in $F_{0}$ for each $t\in (0,1)$.

Fix any $t\in (0,1)$. Note that $g_2, g_3, \eta_1\in\mathbb{R}$ for $\tau=\frac{1}{2}+ib$ with $b>0$. Since $g_3(\frac{1}{2}+\frac{1}{2}i)=0$, we have
\begin{equation}\label{ht12}h_t=g_2^2-18t\eta_1g_3>0\quad \text{at }\tau=\tfrac{1}{2}+\tfrac{1}{2}i.\end{equation}
Since $g_2(\frac{1}{2}+\frac{\sqrt{3}}{2}i)=0$, $g_3(\frac{1}{2}+\frac{\sqrt{3}}{2}i)>0$ and \begin{equation}\label{eta1-rho}\eta_1(\tfrac{1}{2}+\tfrac{\sqrt{3}}{2}i)=\tfrac{2\pi}{\sqrt{3}}\quad\text{(see e.g. \cite[(4.1)]{CL-E2})},\end{equation} we have
\begin{equation}\label{ht121}h_t=g_2^2-18t\eta_1g_3<0\quad \text{at }\tau=\tfrac{1}{2}+\tfrac{\sqrt{3}}{2}i.\end{equation}
Together with (\ref{infty-be-t}) that $\lim_{b\to +\infty}h_t(\frac{1}{2}+ib)>0$, we obtain the existence of $\frac{1}{2}<b_{1,t}<\frac{\sqrt{3}}{2}<b_{2,t}$ such that
\[h_t(\tfrac{1}{2}+i b_{1,t})=h_t(\tfrac{1}{2}+i b_{2,t})=0.\]
Therefore, $\tfrac{1}{2}+i b_{1,t}, \tfrac{1}{2}+i b_{2,t}$ give all the zeros of $h_t(\tau)$ in $F_0$ and are all simple. Moreover, $\lim_{t\to 0}b_{k,t}=\frac{\sqrt{3}}{2}$ for $k=1,2$. The proof is complete.
\end{proof}

We are in the position to prove Theorem \ref{uniquezero}.

\begin{proof}[Proof of Theorem \ref{uniquezero}] First we claim that $h_1=g_2^2-18\eta_1g_3$ has a unique zero on the line $\{\tau |\tau=\frac{1}{2}+ib\}\cap F_0$. This assertion was already proved in \cite{SS} as pointed out in Remark \ref{rmk-ss}. Here we give an alternative proof of this assertion.
Consider $\tau=\frac{1}{2}+ib$ with $b\geq\frac{1}{2}$. Then \cite[Theorem 1.5]{CL-E2} says that
\[\frac{d}{db}\eta_1(\tfrac{1}{2}+ib)=\tfrac{-1}{24\pi}(12\eta_1^2-g_2)(\tfrac{1}{2}+ib)<0.\]
Similarly as (\ref{ht12}), we have $h_1(\frac{1}{2}+\frac{1}{2}i)>0$. Suppose $h_1=g_2^2-18\eta_1g_3=0$ at some $\tau=\frac{1}{2}+ib$ with $b>\frac{1}{2}$, then it follows from
(\ref{2derivative-1}) and $g_3(\frac{1}{2}+ib)>0$ that
\begin{align*}\frac{d}{db}h_1(\tfrac{1}{2}+ib)&=i(g_2^2-18\eta_1 g_3)'\\
&=-\tfrac{21}{4\pi}g_3(\tfrac{1}{2}+ib)(12\eta_1^2-g_2)(\tfrac{1}{2}+ib)<0.\end{align*}
Therefore, $h_1=g_2^2-18\eta_1g_3$ has at most one zero on the line $\{\tau |\tau=\frac{1}{2}+ib, b\geq \frac{1}{2}\}=\{\tau |\tau=\frac{1}{2}+ib\}\cap F_0$. Together with $h_1(\frac{1}{2}+\frac{\sqrt{3}}{2}i)<0$ (see (\ref{ht121})), we conclude that there is $b_{\infty}\in (\frac{1}{2}, \frac{\sqrt{3}}{2})$ such that $\tau_{\infty}:=\frac{1}{2}+i b_{\infty}$ is the unique zero of $h_1$ on the line $\{\tau |\tau=\frac{1}{2}+ib\}\cap F_0$.

Suppose $h_1=g_2^2-18\eta_1g_3$ has a zero $\tilde{\tau}\in F_0\setminus\{\tau |\tau=\frac{1}{2}+ib\}$, then $\tilde{\tau}\in \mathring{F}_{0}$ by Lemma \ref{lemma-t10}.
Consequently, for $t<1$ with $|1-t|$ small enough, $h_t$ has a zero $\tau_t$ close to $\tilde{\tau}$, which is a contradiction with Lemma \ref{lemma-ht}. This shows that $g_2^2-18\eta_1g_3\neq 0$ in $F_0\setminus\{\tau |\tau=\frac{1}{2}+ib\}$ and so $\tau_{\infty}$ is the unique zero of $g_2^2-18\eta_1g_3$ in $F_0$. The proof is complete.
\end{proof}

\begin{remark} Note that $g_2^2-18\eta_1g_3=0$ at $\tau=\frac{1}{2}+i\infty$; see (\ref{infty-be-t}) with $t=1$.
Recall the two zeros $\tfrac{1}{2}+i b_{1,t}, \tfrac{1}{2}+i b_{2,t}$ of $h_t(\tau)$ in $F_0$ for $t\in (0,1)$. Together with Theorem \ref{uniquezero}, we conclude that
\[\lim_{t\to 1}b_{1,t}=b_{\infty}=0.6341269863,\qquad \lim_{t\to 1} b_{2,t}=+\infty.\]
\end{remark}

\section{Two critical points in $\gamma(F_0)$ for $c\neq 0$}
\label{twocriticalp}

This section is devoted to prove Theorem \ref{thm-number}. By Theorem \ref{uniquezero}, it suffices for us to consider $\gamma(F_0)$ with $c\neq 0$. For this purpose,
we define holomorphic function $f_C(\tau)$ for each $C\in\mathbb{R}$ by
\begin{align}\label{fC-exp}f_C(\tau):=&(C-\tau)g_2(\tau)^2-18(C\eta_1(\tau)-\eta_2(\tau))g_3(\tau)\nonumber\\
=&(C-\tau)(g_2(\tau)^2-18\eta_1(\tau)g_3(\tau))-36\pi i g_3(\tau).\end{align}
Note that $f_C(\tau)=0$ is equivalent to
\begin{equation}\label{Cphi}C=\phi(\tau):=\tau+\frac{36\pi i g_3(\tau)}{(g_2^2-18\eta_1 g_3)(\tau)}.\end{equation}

First we need to prove the following result concerning the zero structure of $f_{C}(\tau)$ in $F_{0}$, which is of independent interest.

\begin{theorem}
\label{Unique-pole}Let $C\in\mathbb{R}$ and recall $\tau_{\infty}=\frac{1}{2}+ib_{\infty}$ with $b_{\infty}\in(\frac{1}{2},\frac{\sqrt{3}}{2})$ in Theorem \ref{uniquezero}. Then the following statements hold.
\begin{itemize}
\item[(1)] For any $C\in \mathbb{R}\backslash \{0,1\}$, $f_{C}(\tau)$
has exactly two different zeros in $F_{0}$, which both belong to $\mathring{F}_{0}$ and are simple.
\item[(2)] $\tau_0:=\frac{1}{1-\tau_{\infty}}\in \mathring{F}_{0}$ is the unique zero of $f_0(\tau)$ in $F_0$. Clearly $\operatorname{Re}\tau_0>\frac{1}{2}$ and $|\tau_0-1|=1$.
\item[(3)] $\tau_1:=\frac{\tau_{\infty}-1}{\tau_{\infty}}\in \mathring{F}_{0}$ is the unique zero of $f_1(\tau)$ in $F_0$. Clearly $\tau_1=1-\overline{\tau_0}$, $\operatorname{Re}\tau_1<\frac{1}{2}$ and $|\tau_1|=1$.
\end{itemize}
\end{theorem}

We will see that Theorem \ref{thm-number}-(2) is a consequence of Theorem \ref{Unique-pole}. The proof of Theorem \ref{Unique-pole} (2)-(3) is easy by applying Theorem \ref{uniquezero}.

\begin{proof}[Proof of Theorem \ref{Unique-pole} (2)-(3)] Suppose $f_0(\tau)=0$ for some $\tau\in F_0$. Then $\tau':=\frac{\tau-1}{\tau}\in F_0$ and it follows from (\ref{c-c-5}) that
\[(g_2^2-18\eta_1g_3)(\tau')=-\tau^7f_0(\tau)=0.\]
So Theorem \ref{uniquezero} implies $\tau'=\tau_{\infty}$, i.e. $\tau=\frac{1}{1-\tau_{\infty}}=:\tau_0$. This proves (2).

Similarly, suppose $f_1(\tau)=0$ for some $\tau\in F_0$. Then $\tau':=\frac{1}{1-\tau}\in F_0$ and it follows from (\ref{cc-c-5}) that
\[(g_2^2-18\eta_1g_3)(\tau')=(1-\tau)^7f_1(\tau)=0.\]
Again this infers $\tau'=\tau_{\infty}$, i.e. $\tau=\frac{\tau_{\infty}-1}{\tau_{\infty}}=:\tau_1$. This proves (3).
\end{proof}

The proof of Theorem \ref{Unique-pole}-(1) is much more delicate, and we need to establish several lemmas first.

\begin{lemma}\label{simplezero-F0}
For any $C\in\mathbb{R}$, any zero of $f_C(\tau)$ in $F_0$ is simple.
\end{lemma}

\begin{proof}
Suppose $f_C(\tau)=0$ for some $C\in\mathbb{R}$ and $\tau\in F_0$. Then (\ref{Cphi}) gives
\[C-\tau=\frac{36\pi i g_3(\tau)}{(g_2^2-18\eta_1 g_3)(\tau)} \quad \text{and } (g_2^2-18\eta_1 g_3)(\tau)\neq 0.\]
If $g_2(\tau)=0$, then $\tau=\frac{1}{2}+\frac{\sqrt{3}}{2}i$ and (\ref{eta1-rho}) gives $\eta_1(\tau)=\frac{2\pi}{\sqrt{3}}$. Consequently,
\[C=\tau-\tfrac{2\pi i}{\eta_1(\tau)}=\tfrac{1}{2}-\tfrac{\sqrt{3}}{2}i\notin\mathbb{R},\]
a contradiction. So $g_2(\tau)\neq 0$. Applying (\ref{d-eta1})-(\ref{2derivative}) and (\ref{fC-exp}), we have
\begin{align}\label{fcde}
f_C'(\tau)&=-(g_2^2-18\eta_1 g_3)+(C-\tau)(g_2^2-18\eta_1 g_3)'-36\pi i g_3'\nonumber\\
&=-7(g_2^2-18\eta_1 g_3)-\frac{63  g_3(4\eta_1 g_2^2-36\eta_1^2g_3-3g_2g_3)}{g_2^2-18\eta_1 g_3}\nonumber\\
&=-7\frac{g_2(\tau)(g_2^3-27g_3^2)(\tau)}{(g_2^2-18\eta_1 g_3)(\tau)}\neq 0.
\end{align}
This completes the proof.
\end{proof}

\begin{lemma}
\label{lemma-10}For any $C\in \mathbb{R}$, $f_{C}(\tau
)\not =0$ for $\tau \in \partial F_{0}\cap \mathbb{H}$.
\end{lemma}

\begin{proof}
Suppose $f_{C}(\tau)=0$ for some $C\in \mathbb{R}$ and $\tau \in \partial F_{0}\cap \mathbb{H}$. It follows from Theorem \ref{Unique-pole} (2)-(3) that $C\in\mathbb{R}\setminus\{0,1\}$.

\textbf{Case 1.} $\tau \in i\mathbb{R}_{>0}$.

Then it follows from (\ref{Cphi}) and (\ref{g3eta1})-(\ref{ib-real}) that
\[C=\tau+\frac{36\pi i g_3(\tau)}{(g_2^2-18\eta_1 g_3)(\tau)}\in i\mathbb{R},\]
a contradiction with our assumption $C\in\mathbb{R}\setminus \{0\}$.

\textbf{Case 2.} $|\tau-\frac{1}{2}|=\frac{1}{2}$.

Then $\tau^{\prime}:=\frac{\tau}{1-\tau}\in i\mathbb{R}_{>0}$. Define
$C^{\prime}:=\frac{C}{1-C}\in \mathbb{R}\setminus\{0\}$. Applying (\ref{c-5}),
a direct computation leads to
\[
f_{C^{\prime}}(\tau^{\prime})=\frac{(1-\tau)^{7}}{1-C}f_{C}(\tau)=0.
\]
Then we obtain a contradiction as Case 1.

\textbf{Case 3.} $\tau \in1+i\mathbb{R}_{>0}$.

Then $\tau^{\prime}:=\tau-1\in i\mathbb{R}_{>0}$. Define $C^{\prime}:=C-1\in\mathbb{R}\setminus\{0\}$. By
using $g_{2}(\tau^{\prime})=g_{2}(\tau)$, $g_3(\tau')=g_3(\tau)$ and%
\begin{equation}
\eta_{1}(\tau^{\prime})=\eta_{1}(\tau),\text{ }\eta_{2}(\tau^{\prime}%
)=\eta_{2}(\tau)-\eta_{1}(\tau), \label{c-6}%
\end{equation}
we easily obtain $f_{C^{\prime}}(\tau^{\prime})=f_{C}(\tau)=0$, again a
contradiction as Case 1.

The proof is complete.
\end{proof}

\begin{lemma}
\label{infinity-behavior}Let $C\in \mathbb{R}$ and $\tau\in F_0$. Then
\begin{equation}\label{infty-be}
\lim_{\tau\to\infty}f_C(\tau)=-\tfrac{32}{3}\pi^7 i;
\end{equation}
\begin{equation}\label{0-be}
\lim_{\tau\to 0}f_C(\tau)= \infty\;\text{for $C\neq 0$};\quad  \lim_{\tau\to 1}f_C(\tau) =\infty \;\text{for $C\neq 1$}.
\end{equation}
\end{lemma}

\begin{proof} Denote $q=e^{2\pi i\tau}$ as before.
It follows from the $q$-expansions (\ref{q-exp-eta1})-(\ref{q-exp-g3}) that for $F_0\ni \tau\to\infty$,
\begin{equation}\label{g2g3-infity}g_2(\tau)^2-18\eta_1(\tau) g_3(\tau)=1792\pi^8 q(1+66q)+O(|q|^3).\end{equation}
Inserting (\ref{g2g3-infity}) into (\ref{fC-exp}) and using $\tau q\to 0$, we immediately obtain (\ref{infty-be}) for all $C\in\mathbb{R}$.

The asymptotics at the cusps $\{0,1\}$ can be obtained
via (\ref{infty-be}). First we consider $C\in\mathbb{R}\setminus\{0\}$. Let $\tau^{\prime}:=\frac{\tau-1}{\tau}$ and
$C^{\prime}:=\frac{C-1}{C}\in \mathbb{R}$. Applying (\ref{c-c-5}),
a direct computation leads to%
\[
f_{C^{\prime}}(\tau^{\prime})=\frac{\tau^{7}}{C}f_{C}(\tau).
\]
When $\tau \in F_{0}$ and $\tau \rightarrow0$, we have $\tau'=\frac{\tau-1}{\tau}\in
F_{0}$ and $\tau'\rightarrow \infty$. So
\begin{align}
f_{C}(\tau)=\frac{C}{\tau^{7}}f_{C^{\prime}}(\tau^{\prime})=\frac{C}{\tau^{7}}(-\tfrac{32}{3}\pi^7 i+o(1))\quad\text{as $\tau\to 0$}.\label{ayp-8}%
\end{align}
This proves the first formula in (\ref{0-be}).

Finally we consider $C\in\mathbb{R}\setminus\{1\}$. Let $\tau^{\prime}:=\frac{1}{1-\tau}$ and
$C^{\prime}:=\frac{1}{1-C}\in \mathbb{R}$. Recalling (\ref{cc-c-5}),
a direct computation leads to
\begin{equation}\label{eqw}
f_{C^{\prime}}(\tau^{\prime})=f_{\frac{1}{1-C}}(\tfrac{1}{1-\tau})=\frac{(1-\tau)^{7}}{1-C}f_{C}(\tau).
\end{equation}
When $\tau \in F_{0}$ and $\tau \rightarrow1$, we have $\tau'=\frac{1}{1-\tau}\in
F_{0}$ and $\tau'\rightarrow \infty$.
So
\begin{align*}
f_{C}(\tau)=\frac{1-C}{(1-\tau)^{7}}f_{C^{\prime}}(\tau^{\prime})=\frac{1-C}{(1-\tau)^7}(-\tfrac{32}{3}\pi^7 i+o(1))\quad\text{as $\tau\to 1$}.
\end{align*}
This proves the second formula in (\ref{0-be}).
\end{proof}

Define
\begin{equation}\label{R-divide}\Gamma_1:=(-\infty, 0),\quad \Gamma_2:=(0,1),\quad \Gamma_3:=(1,+\infty).\end{equation}

\begin{lemma}
\label{yl-1}Fix $k\in \{1,2,3\}$. Then the number of zeros (up to multiplicity) of $f_C(\tau)$ in $F_{0}$ is a constant for $C\in \Gamma_{k}$.
\end{lemma}
\begin{proof}
Since $C\in \Gamma_{k}$, Lemma \ref{infinity-behavior} implies
\[
f_C(\tau)\not \rightarrow 0\;\text{ as }\;F_{0}\ni \tau \rightarrow
\infty,0,1\;\text{respectively}.
\]
Together with Lemma \ref{lemma-10} that $f_C
(\tau)\not =0$ on $\partial F_{0}\cap \mathbb{H}$, we conclude as Lemma \ref{lemma-ht} that the number of zeros of $f_C
(\tau)$ in $F_{0}$ is a constant for $C\in \Gamma_{k}$.
\end{proof}

Now we are in the position to prove Theorem \ref{Unique-pole}-(1) and hence completes the proof of Theorem \ref{Unique-pole}.

\begin{proof}[Proof of Theorem \ref{Unique-pole}-(1)] Recall the meromorphic function $\phi(\tau)$ in (\ref{Cphi}) and the fact that $f_{C}(\tau)=0$ is equivalent to $\phi(\tau)=C$. Let $C\in\mathbb{R}\setminus\{0,1\}$.

{\bf Step 1.} We prove that in $F_0$, for $|C|>1$ large enough, $f_C(\tau)$ has a unique zero $\tau_1(C)$ near $\infty$, i.e. satisfies $\tau_1(C)\to \infty$ as $|C|\to +\infty$.

By using the $q$-expansions (\ref{q-exp-eta1})-(\ref{q-exp-g3}) and (\ref{g2g3-infity}), we obtain for $F_0\ni \tau=a+bi\to\infty$ that
\begin{align}\label{phi-infty}
\phi(\tau)  &  =\tau+\frac{i}{168\pi}q^{-1}-\frac{95i}{28\pi}+O(|q|)\nonumber\\
&  =a+\frac{\sin2\pi a}{168\pi}e^{2\pi b}+i\bigg(  b+\frac{\cos2\pi a}{168\pi
}e^{2\pi b}-\frac{95}{28\pi}\bigg)  +O(e^{-2\pi b}).
\end{align}
Therefore, when $C\rightarrow\pm\infty$, it is easy to
prove the existence of \emph{a unique $\tau_{1}(C)=a_{1}(C)+ib_{1}(C)\in F_{0}$}
such that $C=\phi(\tau_{1}(C))$ and $\tau_1(C)\to \infty$ as $|C|\to +\infty$. Furthermore,
\begin{equation}\label{tau1}
b_{1}(C)\rightarrow+\infty,\text{ }a_{1}(C)\rightarrow \left \{
\begin{array}
[c]{c}%
1/4+\text{ \ if }C\rightarrow+\infty,\\
3/4-\text{ \ if }C\rightarrow-\infty.
\end{array}
\right.
\end{equation}
This shows that $\tau_1(C)$ is the unique zero of $f_C(\tau)$ in $F_0$ that is near $\infty$ for $|C|>1$ large enough.

{\bf Step 2.} We prove that
\begin{equation}\label{phi-01}\lim_{F_0\ni \tau\to \infty}\phi(\tau)=\infty,\quad\lim_{F_0\ni \tau\to 0}\phi(\tau)=0,\quad \lim_{F_0\ni \tau\to 1}\phi(\tau)=1.\end{equation}
Consequently, for $|C|>1$ large enough, $f_C(\tau)$ has no zeros in $F_0$ which are near the cusps $\{0,1\}$.

Note that $\lim_{F_0\ni \tau\to \infty}\phi(\tau)=\infty$ is just (\ref{phi-infty}).
When $\tau \in F_{0}$ and $\tau \rightarrow0$, we have $\tau':=\frac{\tau-1}{\tau}\in
F_{0}$ and $\tau'\rightarrow \infty$. Applying the first equality of (\ref{ayp-8-t}) and (\ref{infty-be}) (note $f_{1,C}=f_{C}$) leads to
\[
\lim_{\tau\to 0}\phi(\tau)=\lim_{\tau\to 0}\frac{36\pi i g_3(\tau)}{h_1(\tau)}
=\lim_{\tau\to 0}\frac{36\pi i \tau^{-6}g_3(\tau')}{\tau^{-7}f_{1}(\tau')}=0.
\]
When $\tau \in F_{0}$ and $\tau \rightarrow1$, we have $\tau'=\frac{1}{1-\tau}\in
F_{0}$ and $\tau'\rightarrow \infty$. Similarly, it follows from the first equality of (\ref{eq-h-t}) that
\begin{align*}
\lim_{\tau\to 1}\phi(\tau)
=1+\lim_{\tau\to 1}\frac{36\pi i (1-\tau)^{-6}g_3(\tau')}{(\tau-1)^{-7}f_0(\tau')}=1.
\end{align*}

{\bf Step 3.} We prove that in $F_0$, for $|C|>1$ large enough, $f_C(\tau)$ has a unique zero $\tau_2(C)$ that are uniformly bounded away from the cusps $\{0,1,\infty\}$ as $|C|\to +\infty$. Furthermore,
\begin{equation}
\lim_{|C|\to +\infty}\tau_2(C)=\tau_{\infty}.
\end{equation}

Recall Theorem \ref{uniquezero} that in $F_0$, $g_2^2-18\eta_1g_3$ has a unique and simple zero $\tau_{\infty}=\frac{1}{2}+ib_{\infty}\in \mathring{F}_{0}$. Clearly $g_3(\tau_{\infty})\neq 0$.
Then it follows from
\[C=\phi(\tau)=\tau+\frac{36\pi i g_3(\tau)}{(g_2^2-18\eta_1 g_3)(\tau)}\Leftrightarrow f_C(\tau)=0\]
that $f_C(\tau)$ has a unique simple zero $\tau_2(C)$ near $\tau_{\infty}$ for $|C|>1$ large enough.

Suppose that as $|C|\to +\infty$, $\tau(C)\in F_0$ is any zero of $f_C(\tau)$ that are uniformly bounded away from the cusps $\{0,1,\infty\}$. We may assume $\tau(C)\to \tau(\infty)\in F_0$ up to a subsequence. Then $C=\phi(\tau(C))$ implies $\infty=\phi(\tau(\infty))$, i.e. $(g_2^2-18\eta_1g_3)(\tau(\infty))=0$. It follows from Theorem \ref{uniquezero} that $\tau(\infty)=\tau_{\infty}$ and so $\tau(C)=\tau_2(C)$ for $|C|$ large enough. This proves Step 3.

{\bf Step 4.} We complete the proof.

Steps 1-3 show that for $|C|>1$ large enough, $f_C(\tau)$ has exactly two zeros $\tau_1(C)$ and $\tau_2(C)$ in $F_0$, one of which is near $\infty$ and the other is near $\tau_{\infty}$.
Together with Lemmas \ref{simplezero-F0} and \ref{yl-1}, we conclude that $f_{C}(\tau)$ has exactly two different zeros in $F_0$ for $C\in (1,+\infty)\cup (-\infty, 0)$. Finally, since $\frac{1}{1-\tau}\in F_0$ if and only if $\tau\in F_0$, it follows from (\ref{eqw}) that $f_{C}(\tau)$ has also exactly two different zeros in $F_0$ for $C\in (0,1)$. This completes the proof.
\end{proof}

Now we apply Theorem \ref{Unique-pole} to prove our first main result Theorem \ref{thm-number}.

\begin{proof}[Proof of Theorem \ref{thm-number}]
Theorem \ref{thm-number}-(1) is just Theorem \ref{uniquezero}.
So it suffices to consider $c\neq 0$.
Given $\gamma=
\bigl(\begin{smallmatrix}a & b\\
c & d\end{smallmatrix}\bigr)
\in \Gamma_{0}(2)/\{\pm I_2\}$ with $c\not =0$. Write $\tau^{\prime}%
=\gamma \cdot \tau=\frac{a\tau+b}{c\tau+d}$ with $\tau \in F_{0}$. By using $g_{3}(\tau^{\prime})=(c\tau+d)^{6}g_{3}(\tau)$ and
\begin{equation}
\eta_{1}(\tau^{\prime})=(c\tau+d)(c\eta_{2}(\tau)+d\eta_{1}(\tau)),\text{
\ }g_{2}(\tau^{\prime})=(c\tau+d)^{4}g_{2}(\tau), \label{c-38}%
\end{equation}
we have%
\begin{align}\label{g3-fC}
&g_2(\tau')^2-18\eta_{1}(\tau')g_{3}(\tau') \\
=&-c(c\tau
+d)^{7}\left[ g_{2}(\tau)(-\tfrac{d}%
{c}-\tau)-18(-\tfrac{d}{c}\eta_{1}(\tau)-\eta_{2}(\tau))g_3(\tau)\right] \nonumber\\
=&-c(c\tau+d)^{7}f_{\frac{-d}{c}}(\tau).\nonumber
\end{align}
Since $\frac{-d}{c}\in \mathbb{Q}\setminus\{0,1\}$, Theorem \ref{Unique-pole} shows that $f_{\frac{-d}{c}}(\tau)$ has exactly two different zeros, denoted by $\tau_1(\frac{-d}{c})$ and $\tau_2(\frac{-d}{c})$, in $F_0$.
Consequently, we conclude from (\ref{g3-fC}) that $g_2^2-18\eta_1g_3$ has exactly two different zeros $\frac{a\tau_1(\frac{-d}{c})+b}{c\tau_1(\frac{-d}{c})+d}$ and $\frac{a\tau_2(\frac{-d}{c})+b}{c\tau_2(\frac{-d}{c})+d}$ in $\gamma(F_0)$.
The proof is complete.
\end{proof}

\section{The three smooth curves}

\label{smoothcurves}

This section is devoted to proving Theorem \ref{Location} concerning the distribution of the critical points. For this purpose, we need to give the precise definition of the three curves in Theorem \ref{Location}. First we study some basic properties of the zeros of $f_{C}(\tau)$.

\begin{lemma}\label{line12}For any $C\in\mathbb{R}$, $f_C(\tau)\neq 0$ for any $\tau=\frac{1}{2}+ib$ with $b\geq \frac{1}{2}$.
\end{lemma}

\begin{proof} Define a function $\varphi: [\frac{1}{2}, +\infty)\to \mathbb{R}\cup \{\infty\}$ by
\[\varphi(b):=b+\frac{36\pi g_3(\frac{1}{2}+ib)}{(g_2^2-18\eta_1g_3)(\frac{1}{2}+ib)}.\]
We only need to prove
\begin{equation}\label{varphib}
\varphi(b)\neq 0\quad \text{for any } b\geq \tfrac{1}{2}.
\end{equation}
Once (\ref{varphib}) holds, we have $\phi(\frac{1}{2}+ib)=\frac{1}{2}+i\varphi(b)\notin \mathbb{R}$ and so $f_C(\frac{1}{2}+i b)\neq 0$ for all $C\in\mathbb{R}$ and $b\geq \frac{1}{2}$.

It is well known that $g_3(\frac{1}{2}+\frac{1}{2}i)=0$ and $g_3(\frac{1}{2}+ib)>0$ for all $b>\frac{1}{2}$.
Recalling Theorem \ref{uniquezero} that $b_{\infty}\in (\frac{1}{2},\frac{\sqrt{3}}{2})$ is the unique zero of $(g_2^2-18\eta_1g_3) (\frac{1}{2}+ib)$ on $[\frac{1}{2}, +\infty)$ and
\[(g_2^2-18\eta_1g_3)(\tfrac{1}{2}+ib)\begin{cases} >0\quad\text{if}\; b\in [\frac{1}{2}, b_{\infty}),\\
<0\quad \text{if}\; b>b_{\infty},
\end{cases}\]
we have $\varphi(b)>0$ for $b\in [\frac{1}{2}, b_{\infty})$, $\varphi(b_{\infty})=\infty$ and $\varphi(b)<0$ for $b-b_{\infty}>0$ small.
For $b>b_{\infty}$, we easily derive from (\ref{d-eta1})-(\ref{2derivative}) that (similar as (\ref{fcde}))
\begin{align*}
\varphi'(b)&=1+36\pi\frac{(g_2^2-18\eta_1 g_3)\frac{d}{db}g_3-g_3\frac{d}{db}(g_2^2-18\eta_1 g_3)}{(g_2^2-18\eta_1 g_3)^2}(\tfrac{1}{2}+ib)\\
&=7\frac{g_2(g_2^3-27g_3^2)}{(g_2^2-18\eta_1 g_3)^2}(\tfrac{1}{2}+ib).
\end{align*}
Together with $(g_2^3-27g_3^2)(\tfrac{1}{2}+ib)<0$, $g_2(\tfrac{1}{2}+\frac{\sqrt{3}}{2}i)=0$ and
\[g_2(\tfrac{1}{2}+ib)\begin{cases}<0 \;\text{if }b\in (b_\infty, \tfrac{\sqrt{3}}{2}),\\
>0\;\text{if }b>\tfrac{\sqrt{3}}{2},\end{cases}\]
we conclude that
\[\max_{b>b_{\infty}}\varphi(b)=\varphi(\tfrac{\sqrt{3}}{2})
=\tfrac{\sqrt{3}}{2}-\tfrac{2\pi}{\eta_1(\frac{1}{2}+\frac{\sqrt{3}}{2}i)}=-\tfrac{\sqrt{3}}{2},\]
where (\ref{eta1-rho}) is used.
This proves (\ref{varphib}).
\end{proof}

Motivated by Lemma \ref{line12}, we define
\[F_0^{<}:=\{\tau\in \mathring{F}_0\,|\,\operatorname{Re}{\tau}<\tfrac{1}{2}\},\;
F_0^{>}:=\{\tau\in \mathring{F}_0\,|\,\operatorname{Re}{\tau}>\tfrac{1}{2}\}.\]

\begin{lemma}\label{lemma-1-1}
For any $C\in\mathbb{R}\setminus\{0,1\}$, $f_{C}(\tau)$ has exactly one zero in each of $F_0^{<}$ and $F_0^{>}$.
\end{lemma}

\begin{proof} Recall (\ref{R-divide}) that
\[\Gamma_1=(-\infty, 0),\quad \Gamma_2=(0,1),\quad \Gamma_3=(1,+\infty).\]
Fix $k\in \{1,2,3\}$.
For any $C\in \Gamma_k$, Theorem \ref{Unique-pole}-(1) shows that $f_C(\tau)$ has exactly two different zeros $\tau_1(C)$ and $\tau_2(C)$ in $F_0$, which are both simple and locate in $F_0^{<}\cup F_0^{>}$ by applying Lemmas \ref{simplezero-F0}-\ref{lemma-10} and \ref{line12}. By the implicit function theorem and renaming $\tau_1(C)$ and $\tau_2(C)$ if necessary, we obtain
\begin{align}\label{fc4-2}&\text{\it two smooth functions}\;\Gamma_k\ni C \mapsto \tau_j(C)\in F_0^{<}\cup F_0^{>},\; j=1,2, \\ &\text{\it the images of which have no intersections.}\nonumber\end{align}
If $k\in \{1,2\}$, we easily see from Theorem \ref{Unique-pole}-(2), Lemma \ref{simplezero-F0} and (\ref{phi-01}) that
\begin{equation}\label{asy-0-k}
\Big\{\lim_{\Gamma_k\ni C\to 0}\tau_1(C), \lim_{\Gamma_k\ni C\to 0}\tau_2(C)\Big\}
=\{\tau_0, 0\}.
\end{equation}
Since $\tau_0\in F_0^{>}$ and $0\in \partial F_0^{<}$, we conclude that exact one of $\{\tau_1(C), \tau_2(C)\}$ belongs to $F_0^{<}$ (resp. $F_0^{>}$) for any $C\in \Gamma_k$.

Similarly, if $k\in \{2,3\}$, we deduce from Theorem \ref{Unique-pole}-(3), Lemma \ref{simplezero-F0} and (\ref{phi-01}) that
\begin{equation}\label{asy-1-k}
\Big\{\lim_{\Gamma_k\ni C\to 1}\tau_1(C), \lim_{\Gamma_k\ni C\to 1}\tau_2(C)\Big\}
=\{\tau_1, 1\}.
\end{equation}
Since $\tau_1\in F_0^{<}$ and $1\in \partial F_0^{>}$, we conclude that exact one of $\{\tau_1(C), \tau_2(C)\}$ belongs to $F_0^{<}$ (resp. $F_0^{>}$) for any $C\in \Gamma_k$. This completes the proof.
\end{proof}

Lemma \ref{lemma-1-1} leads us to give the following notations.

\medskip

\noindent{\bf Notations:} \emph{For any $C\in\mathbb{R}\setminus\{0,1\}$, we denote by $\tau_{<}(C)$ (resp. $\tau_{>}(C)$) to be the unique zero of $f_{C}(\tau)$ in $F_0^{<}$ (resp. in $F_0^{>}$)}.

\medskip

Since \[
\mathbb{H=}\bigcup_{\gamma \in \Gamma_{0}(2)/\{ \pm I_{2}\}}\gamma (F_{0}),
\] we can restate Theorem \ref{thm-number} as follows.

\begin{theorem}
\label{thm-number-1} Let $\gamma (F_{0})$ be a fundamental domain of $\Gamma_{0}(2)$ with $\gamma=
\bigl(\begin{smallmatrix}a & b\\
c & d\end{smallmatrix}\bigr)
\in \Gamma_{0}(2)/\{ \pm I_{2}\}$. Then
the following statements hold:

\begin{itemize}
\item[(1)] If $c=0$, i.e. $\gamma(F_0)=F_0+m$ with $m=\frac{b}{d}\in \mathbb{Z}$, then $\tau_{\infty}+m$ is the unique zero of $E_6^{\prime}(\tau)$ in $F_{0}+m$.

\item[(2)] If $c\not =0$, then $E_6^{\prime}(\tau)$ has exactly two different zeros $\frac{a\tau_<
(-d/c)+b}{c\tau_<(-d/c)+d}$ and $\frac{a\tau_>
(-d/c)+b}{c\tau_>(-d/c)+d}$ in
the fundamental domain $\gamma(F_{0})$ of $\Gamma_{0}(2)$. In particular,%
\begin{align*}
&(\tau_{\infty}+\mathbb{Z})\cup\\
&\left \{  \left.  \frac{a\tau_{<}(\tfrac{-d}{c})+b}{c\tau_{<}(\tfrac{-d}{c}%
)+d}, \frac{a\tau_{>}(\tfrac{-d}{c})+b}{c\tau_{>}(\tfrac{-d}{c}%
)+d}\right \vert
\bigl(\begin{smallmatrix}a & b\\
c & d\end{smallmatrix}\bigr)
\in \Gamma_{0}(2)/\{ \pm I_{2}\},\,c\not =0\right \}
\end{align*}
gives rise to all the zeros of $E_6^{\prime}(\tau)$ in $\mathbb{H}$.
\end{itemize}
\end{theorem}

Now we can give the precise definition of the three curves. Recall that
for each $k\in \{1,2,3\}$, the proof of Lemma \ref{lemma-1-1} implies that
\begin{align*}\Gamma_k\ni C\mapsto \tau_<(C)\in F_0^{<}\;\text{\it is smooth};\\
 \Gamma_k\ni C\mapsto \tau_>(C)\in F_0^{>}\;\text{\it is smooth}.\end{align*}
It follows from Step 1 in the proof of Theorem \ref{Unique-pole}-(1) (see particularly (\ref{tau1})), $\lim_{C\to +\infty}\tau_>(C)\neq\tfrac{1}{4}+i\infty$ and $\lim_{C\to -\infty}\tau_<(C)\neq\tfrac{3}{4}+i\infty$ that
\begin{equation}\label{Cccinf}\lim_{C\to +\infty}\tau_<(C)=\tfrac{1}{4}+i\infty,\quad \lim_{C\to -\infty}\tau_>(C)=\tfrac{3}{4}+i\infty,\end{equation}
and then $\lim_{C\to +\infty}\tau_>(C)\notin\{0,1,\infty\}$ and $\lim_{C\to -\infty}\tau_<(C)\notin\{0,1,\infty\}$, where we also used (\ref{phi-01}). Since $C=\phi(\tau_{\lessgtr}(C))$ and $\phi(\tau)=\infty$ if and only if $(g_2^2-18\eta_1g_3)(\tau)=0$, we finally conclude from Theorem \ref{uniquezero} that
\begin{equation}\label{curve3-1}\lim_{C\to +\infty}\tau_>(C)=\tau_{\infty}=\lim_{C\to -\infty}\tau_<(C).\end{equation}

On the other hand, we see from (\ref{asy-0-k})-(\ref{asy-1-k}) that
{\allowdisplaybreaks
\begin{align}\label{curve1-1}
\lim_{(0,1)\ni C\to 1}\tau_{<}(C)&=\tau_1=\lim_{(1,+\infty)\ni C\to 1}\tau_{<}(C),\\
\lim_{(0,1)\ni C\to 1}\tau_{>}(C)&=1=\lim_{(1,+\infty)\ni C\to 1}\tau_{>}(C),\nonumber\\
\lim_{(-\infty,0)\ni C\to 0}\tau_{<}(C)&=0=\lim_{(0,1)\ni C\to 0}\tau_{<}(C),\nonumber\\
\label{curve2-1}\lim_{(-\infty,0)\ni C\to 0}\tau_{>}(C)&=\tau_0=\lim_{(0,1)\ni C\to 0}\tau_{>}(C).
\end{align}
}%
Since $f_1(\tau_1)=0$, by defining $\tau_<(1):=\tau_1$, it follows from the implicit function theorem, (\ref{Cccinf}) and (\ref{curve1-1})-(\ref{curve2-1}) that
\begin{align*}(0,+\infty)\ni C\mapsto \tau_<(C)\in F_0^{<}\;\text{\it is smooth},\end{align*}
and so give a curve in $F_0$:
\begin{align}\label{curve1-2}&\mathcal{C}_{2}:=\{\tau_{<}(C)\,|\, C\in (0,+\infty), \tau_<(1)=\tau_1\}\subset F_0^{<}\\
&\text{with}\; \partial \mathcal{C}_{2}=\{0, \tfrac{1}{4}+i\infty\}.\nonumber\end{align}
Similarly, since $f_0(\tau_0)=0$, by defining $\tau_{>}(0):=\tau_0$, we see that
\begin{align*}(-\infty, 1)\ni C\mapsto \tau_>(C)\in F_0^{>}\;\text{\it is smooth},\end{align*}
and so give a curve in $F_0$:
\begin{align}\label{curve2-2}&\mathcal{C}_{3}:=\{\tau_{>}(C)\,|\, C\in (-\infty,1), \tau_{>}(0)=\tau_0\}\subset F_0^{>}\\
&\text{with}\; \partial \mathcal{C}_{3}=\{1, \tfrac{3}{4}+i\infty\}.\nonumber\end{align}
Recall again that $C=\phi(\tau_{\lessgtr}(C))$ and $\phi(\tau_{\infty})=\infty$.
Define
\[\tau(C):=\begin{cases}\tau_{<}(C)\quad\text{if}\;C\in (-\infty, 0),\\
\tau_{\infty}\quad\text{if}\; C=\infty,\\
\tau_{>}(C)\quad \text{if}\; C\in (1, +\infty).\end{cases}\]
Then it follows from (\ref{curve3-1}) that the image gives a curve in $F_0$:
\begin{align}\label{curve3-2}\mathcal{C}_{1}:=&\{\tau(C)\,|\, C\in (-\infty,0)\cup\{\infty\}\cup (1, +\infty)\}\\
=&\{\tau_{<}(C)\,|\, C\in (-\infty,0)\}\cup\{\tau_{\infty}\}\cup\{\tau_{>}(C)\,| \,C\in (1, +\infty)\}\nonumber\\
\text{with}&\; \partial \mathcal{C}_{1}=\{0, 1\}.\nonumber\end{align}

\begin{lemma}\label{pro-curve}  {\ }

\begin{itemize}
\item[(i)] Each of $\mathcal{C}_{1}, \mathcal{C}_2, \mathcal{C}_3$ has no self-intersection, and
any two of them have no intersections.
\item[(ii)]
The curve $\mathcal{C}_{1}$ is symmetric with respect to the line
$\operatorname{Re}\tau=\frac{1}{2}$; $\mathcal{C}_{2}$ and $\mathcal{C}_{3}$ are symmetric with respect to the line $\operatorname{Re}\tau=\frac{1}{2}$.
\end{itemize}
\end{lemma}

\begin{proof}
(i).
By the definitions of the curves and (\ref{fc4-2}), the assertion (i) follows easily from the following simple observation: \emph{For any $C_1\neq C_2$, $f_{C_1}(\tau)$ and $f_{C_2}(\tau)$ have no common zeros}, which is a consequence of the fact that $f_C(\tau)=0$ implies $C=\phi(\tau)$.

(ii). It is easy to prove (see e.g. \cite[Appendix A]{CL-E2}) that $\eta_{1}(1-\bar{\tau
})=\overline{\eta_{1}(\tau)}$, $g_{k}(1-\bar{\tau})=\overline{g_{k}(\tau)}$, $k=2,3$, and
 $\eta_{2}(1-\bar{\tau})=\overline{\eta
_{1}(\tau)}-\overline{\eta_{2}(\tau)}$.
Then for $C\in \mathbb{R}$, we have
\begin{align*}
  &f_{1-C}(1-\bar{\tau})\\
=&(\bar{\tau}-C)g_{2}(1-\bar{\tau})-18\left(  (1-C)\eta_{1}(1-\bar{\tau})-\eta_{2}(1-\bar{\tau})\right)
g_3(1-\bar{\tau})\\
=&-\overline{(C-\tau)g_{2}(\tau)}+18(C\overline{\eta_{1}(\tau)}-\overline{\eta_{2}(\tau)})
\overline{g_3(\tau)}=-\overline{f_{C}(\tau)}.
\end{align*}
Thus, $\tau_1=1-\overline{\tau_0}$ (this was already proved in Theorem \ref{Unique-pole}) and
\[\tau_{<}(1-C)=1-\overline{\tau_{>}(C)},\;\tau_{>}(1-C)=1-\overline{\tau_{<}(C)},\;\forall C\in\mathbb{R}\setminus\{0,1\}.\]
Again together with the definitions (\ref{curve1-2})-(\ref{curve3-2}) of the curves, the assertion (ii) follows. This completes the proof.
\end{proof}

\begin{lemma}\label{pro-smooth}
The three curves $\mathcal{C}_{1}$, $\mathcal{C}_{2}$, $\mathcal{C}_{3}$ are all smooth curves in $F_0$.
\end{lemma}

\begin{proof}
First we prove that $\mathcal{C}_{2}$ is smooth.
Recall $\phi(\tau)$ defined in $(\ref{Cphi})$. It follows from Lemma \ref{uniquezero} that $g_2^2-18\eta_1g_3\neq 0$ in $F_{0}^{<}\cup F_{0}^{>}$, i.e. $\phi(\tau)$ is holomorphic in $F_0^{<}\cup F_{0}^{>}$.
We claim that
\begin{equation}\label{smoot}
\phi'(\tau)\neq 0\quad \forall \tau\in \mathcal{C}_{2}\subset F_0^{<}.
\end{equation}

Fix any $\tau_0\in \mathcal{C}_{2}$. Then
$\tau_0=\tau_{<}(C)$ for some $C\in (0, +\infty)$. In particular, $f_{C}'(\tau_0)\neq 0$ by Lemma \ref{simplezero-F0}. As pointed out before, $f_{C}(\tau_0)=0$ implies $\phi(\tau_0)=C$, namely
\[C-\tau_0=\frac{36\pi ig_3(\tau_0)}{(g_2^2-18\eta_1g_3)(\tau_0)}.\]
Consequently, we see from (\ref{fcde}) that
{\allowdisplaybreaks
\begin{align*}
\phi'(\tau_0)&=1+\frac{36\pi ig_3'}{g_2^2-18\eta_1g_3}(\tau_0)-\frac{36\pi ig_3(g_2^2-18\eta_1g_3)'}{(g_2^2-18\eta_1g_3)^2}(\tau_0)\\
&=1+\frac{36\pi ig_3'}{g_2^2-18\eta_1g_3}(\tau_0)-(C-\tau_0)\frac{(g_2^2-18\eta_1g_3)'}{g_2^2-18\eta_1g_3}(\tau_0)\\
&=\frac{-f_{C}'(\tau_0)}{(g_2^2-18\eta_1g_3)(\tau_0)}\neq 0.
\end{align*}
}%
This proves (\ref{smoot}).

On the other hand, it is easy to see that
\begin{equation}\label{c-phi}\mathcal{C}_{1}\cup\mathcal{C}_{2}\cup\mathcal{C}_{3}=\{  \tau\in F_{0}\left \vert
\operatorname{Im}\phi(\tau)=0 \;\text{or}\;\phi(\tau)=\infty
\right.\}\end{equation}
and $\mathcal{C}_2\subset \{  \tau\in F_{0}^{<}\,|\,
\operatorname{Im}\phi(\tau)=0\}$.
Write $\tau=a+bi$ with $a,b\in\mathbb{R}$. Since\[
\frac{\partial \operatorname{Im}\phi}{\partial a}=\operatorname{Im}\phi^{\prime},\text{ \  \ }\frac{\partial \operatorname{Im}\phi}{\partial
b}=\operatorname{Re}\phi^{\prime},
\]
we see from (\ref{smoot}) that $\mathcal{C}_{2}$ is smooth at any $\tau\in \mathcal{C}_{2}$, namely $\mathcal{C}_{2}$ is a smooth curve. A similar argument or applying Lemma \ref{pro-curve}-(ii) also shows that $\mathcal{C}_3$ is a smooth curve.
Finally, we recall (\ref{eqw}) that
\[f_{\frac{1}{1-C}}(\tfrac{1}{1-\tau})=\frac{(1-\tau)^{7}}{1-C}f_{C}(\tau),\quad C\in \mathbb{R}\setminus\{1\}.\]
Together with $\tau_\infty=\frac{1}{1-\tau_1}$,
\[
\tfrac{1}{1-\tau}\in F_{0}\Leftrightarrow \tau \in F_{0}\text{ and }\tfrac{1}{1-C}\in \mathbb{R}\backslash \{0,1\} \Leftrightarrow C\in \mathbb{R}%
\backslash \{0,1\},
\]
we easily deduce from (\ref{Cccinf})-(\ref{curve2-1}) that
\[\tau_{<}(\tfrac{1}{1-C})=\tfrac{1}{1-\tau_{<}(C)} \;\text{for}\; C\in (1,+\infty),\;\text{i.e. }\tfrac{1}{1-C}\in (-\infty, 0),\]
\[\tau_{>}(\tfrac{1}{1-C})=\tfrac{1}{1-\tau_{<}(C)} \;\text{for}\; C\in (0,1),\;\text{i.e. }\tfrac{1}{1-C}\in (1, +\infty).\]
Consequently,
{\allowdisplaybreaks
\begin{align}\label{C1-C2}\mathcal{C}_{1}
&=\{\tau_{<}(C)\,|\, C\in (-\infty,0)\}\cup\{\tau_{\infty}\}\cup\{\tau_{>}(C)\,| \,C\in (1, +\infty)\}\nonumber\\
&=\{\tfrac{1}{1-\tau_{<}(C)}\,|\,C\in (0, +\infty)\}\nonumber\\
&=\{\tfrac{1}{1-\tau}\,|\,\tau\in \mathcal{C}_2\}.\end{align}
}%
This implies that $\mathcal{C}_1$ is also a smooth curve in $F_0$.
\end{proof}

We are in the position to prove Theorem \ref{Location}.

\begin{proof}[Proof of Theorem \ref{Location}] Recall that $\mathcal{D}_0$ is the collection of those points in $F_0$ which are obtained by
transforming the critical points of $E_{6}(\tau)$ via the M\"{o}bius transformations of $\Gamma_{0}(2)$ action. Thus, we deduce from Theorem \ref{thm-number-1} that
\begin{align*}
\mathcal{D}_0&=\{\tau_{\infty}\}\cup\left \{  \left.  \tau_{<}(\tfrac{-d}{c}), \tau_{>}(\tfrac{-d}{c})\right \vert
\bigl(\begin{smallmatrix}a & b\\
c & d\end{smallmatrix}\bigr)
\in \Gamma_{0}(2)/\{ \pm I_{2}\},\,c\not =0\right \}\\
&\subset\mathcal{C}_1\cup\mathcal{C}_2\cup\mathcal{C}_3=\{\tau_{\infty}\}\cup\left \{  \left.  \tau_{<}(C), \tau_{>}(C)\right \vert
C\in \mathbb{R}\right \}.
\end{align*}
Note that \[\{\tfrac{-d}{c}\;|\;d\in\mathbb{Z},\, c\in 2\mathbb{Z}\setminus\{0\}, \, (c,d)=1\}\]
is dense in $\mathbb{Q}$ and hence dense in $\mathbb{R}$. Therefore, the denseness of $\mathcal{D}_0$ on $\mathcal{C}_1\cup\mathcal{C}_2\cup\mathcal{C}_3$, i.e. the assertion (\ref{cf0}),  follows. The proof is complete by applying Lemmas \ref{pro-curve} and \ref{pro-smooth}.
\end{proof}

\section{Distribution in fundamental domains of $SL(2,\mathbb{Z})$}

\label{distribution-F}

This section is devoted to proving Theorems \ref{thm-number-F} and \ref{Location-F}. Recall that $F$ is the basic fundamental domain of $SL(2,\mathbb{Z})$ defined in (\ref{F-de}) and $\mathcal{C}_{<}, \mathcal{C}_{>}$ are defined in (\ref{C-F}). We need the following important observation.

\begin{lemma}\label{FFF} The following statements hold.
\begin{align}\label{C2-F}&\mathcal{C}_{<}=\mathcal{C}_2\cap F=\{\tau_{<}(C)\,|\, C\in [1,+\infty), \tau_<(1)=\tau_1\}\subset F\cap F_0^{<}\\
&\text{with }\partial\mathcal{C}_{<}=\{\tau_1, \tfrac{1}{4}+i\infty\};\nonumber\end{align}
\begin{align}\label{C3-F}&\mathcal{C}_{>}=\mathcal{C}_3\cap F=\{\tau_{>}(C)\,|\, C\in (-\infty, 0], \tau_>(0)=\tau_0\}\subset F\cap F_0^{>}\\
&\text{with }\partial\mathcal{C}_{>}=\{\tau_0, \tfrac{3}{4}+i\infty\};\nonumber\end{align}
\begin{equation}\label{C1-F}\mathcal{C}_1\cap F=\emptyset.\end{equation}
In particular, $\mathcal{C}_{<}$ and $\mathcal{C}_{>}$ are symmetric with respect to the line $\operatorname{Re}\tau=\frac{1}{2}$.
\end{lemma}

\begin{proof}
Recalling the definition (\ref{curve1-2}) of $\mathcal{C}_2$ that $\mathcal{C}_2\subset F_0^{<}$, we have $\mathcal{C}_2\cap \partial F\subset \{\tau \,|\,|\tau|=1\}$. Suppose $\tau\in \mathcal{C}_2\cap \partial F$, then $|\tau|=1$ and so $\operatorname{Re}\frac{1}{1-\tau}=\frac{1}{2}$. Since the definition (\ref{curve3-2}) of $\mathcal{C}_1$ implies
\[\mathcal{C}_1\cap \{\tau \,|\, \operatorname{Re}\tau=\tfrac{1}{2}\}=\{\tau_\infty\},\]
we conclude from (\ref{C1-C2}) that $\frac{1}{1-\tau}=\tau_\infty$, i.e. $\tau=\frac{\tau_\infty-1}{\tau_\infty}=\tau_1$ by applying Theorem \ref{Unique-pole}. This proves
\[\mathcal{C}_2\cap \partial F=\{\tau_1\}=\{\tau_{<}(1)\}.\]
Consequently, (\ref{C2-F}) follows from the definition (\ref{curve1-2}) of $\mathcal{C}_2$. By (\ref{C2-F}) and Lemma \ref{pro-curve}-(ii), we easily obtain (\ref{C3-F}), the symmetry of $\mathcal{C}_{<}$ and $\mathcal{C}_{>}$ with respect to the line $\operatorname{Re}\tau=\frac{1}{2}$ and
\[\mathcal{C}_3\cap \partial F=\{\tau_0\}=\{\tau_{>}(0)\}.\]
Finally, suppose $\mathcal{C}_1\cap F\neq \emptyset$, then $\mathcal{C}_1\cap (\partial F\setminus \{\infty\})\neq \emptyset$. Since Lemma \ref{pro-curve}-(ii) says that $\mathcal{C}_1$ is symmetric with respect to $\operatorname{Re}\tau=\frac{1}{2}$, there is $\tau\in \mathcal{C}_1$ such that $|\tau-1|=1$. It follows from (\ref{C1-C2}) that $\tilde{\tau}:=\frac{\tau-1}{\tau}\in \mathcal{C}_2$ and $\operatorname{Re}\tilde{\tau}=\frac{1}{2}$, a contradiction with $\mathcal{C}_2\subset F_0^{<}$. Therefore, (\ref{C1-F}) holds.
\end{proof}

As a consequence of Lemma \ref{FFF}, we can restate Theorem \ref{Unique-pole} as follows.

\begin{theorem}
\label{Unique-pole-F}Recall $f_C(\tau)$ defined in (\ref{fC-exp}). Then the following statements hold.
\begin{itemize}
\item[(1)] $f_{C}(\tau)$ has no zeros in $F$ for any $C\in (0,1)$.
\item[(2)] For any $C\in (1,+\infty)$ (resp. $C\in (-\infty, 0)$), $f_{C}(\tau)$
has a unique zero $\tau_{<}(C)$ (resp. $\tau_{>}(C)$) in $F$, which satisfies $\tau_{<}(C)\in \mathring{F}\cap F_0^<$ (resp. $\tau_{>}(C)\in \mathring{F}\cap F_0^>$).
\item[(3)] $\tau_0=\tau_{>}(0)=\frac{1}{1-\tau_{\infty}}$ is the unique zero of $f_0(\tau)$ in $F$. Clearly $|\tau_0-1|=1$.
\item[(4)] $\tau_1=\tau_{<}(1)=\frac{\tau_{\infty}-1}{\tau_{\infty}}$ is the unique zero of $f_1(\tau)$ in $F$. Clearly $|\tau_1|=1$.
\end{itemize}
\end{theorem}

We are in the position to prove Theorems \ref{thm-number-F} and \ref{Location-F}.

\begin{proof}[Proof of Theorem \ref{thm-number-F}]
Let $\gamma (F)$ be a fundamental domain of $SL(2,\mathbb{Z})$ with $\gamma=
\bigl(\begin{smallmatrix}a & b\\
c & d\end{smallmatrix}\bigr)
\in SL(2,\mathbb{Z})/\{ \pm I_{2}\}$.

If $c=0$, then $\gamma (F)=F+m\subset F_0+m$ for some $m\in\mathbb{Z}$. Since Theorem \ref{uniquezero} shows that $\tau_\infty+m\notin F+m$ is the unique zero of $E_6'(\tau)$ in $F_0+m$, we see that $E_6'(\tau)$ has no zeros in $\gamma (F)$.

So it suffices to consider $c\neq 0$.
Write $\tau^{\prime}
=\gamma \cdot \tau=\frac{a\tau+b}{c\tau+d}$ with $\tau \in F$. Then (\ref{g3-fC}) gives
\[(g_2^2-18\eta_{1}g_{3})(\tau')=-c(c\tau+d)^{7}f_{\frac{-d}{c}}(\tau).
\]
If $\frac{-d}{c}\in (0,1)$, Theorem \ref{Unique-pole-F} shows that $f_{\frac{-d}{c}}(\tau)$ has no zeros in $F$.
Consequently, $g_2^2-18\eta_1g_3$ and so $E_6'(\tau)$ has no zeros in $\gamma(F)$. This proves (1).

If $\frac{-d}{c}\in [1,+\infty)$, Theorem \ref{Unique-pole-F} shows that $f_{\frac{-d}{c}}(\tau)$ has a unique zero $\tau_{<}(\frac{-d}{c})$ in $F$. As before, it follows that $\frac{a\tau_{<}(\frac{-d}{c})+b}{c\tau_{<}(\frac{-d}{c})+d}$ is the unique zero of $E_6'(\tau)$ in $\gamma(F)$, and $\frac{a\tau_{<}(\frac{-d}{c})+b}{c\tau_{<}(\frac{-d}{c})+d}\in \partial\gamma(F)$ if and only if $\frac{-d}{c}=1$.
Similarly, if $\frac{-d}{c}\in (-\infty, 0]$, then $\frac{a\tau_{>}(\frac{-d}{c})+b}{c\tau_{>}(\frac{-d}{c})+d}$ is the unique zero of $E_6'(\tau)$ in $\gamma(F)$, and $\frac{a\tau_{>}(\frac{-d}{c})+b}{c\tau_{>}(\frac{-d}{c})+d}\in \partial\gamma(F)$ if and only if $\frac{-d}{c}=0$. This proves (2).
The proof is complete.
\end{proof}

\begin{proof}[Proof of Theorem \ref{Location-F}]
Theorem \ref{Location-F}-(1) follows from Lemma \ref{FFF} with $\tilde{\tau}_<=\tau_1$ and $\tilde{\tau}_>=\tau_0$. Recall that $\mathcal{D}$ is the collection of those points in $F$ obtained by
transforming the critical points of $E_{6}(\tau)$ via the M\"{o}bius transformations of $SL(2,\mathbb{Z})$ action. It follows from Theorem \ref{thm-number-F} and Lemma \ref{FFF} that
\begin{align*}
\mathcal{D}&=\left \{  \left.  \tau_{<}(\tfrac{-d}{c})\right \vert
\bigl(\begin{smallmatrix}a & b\\
c & d\end{smallmatrix}\bigr)
\in SL(2,\mathbb{Z})/\{ \pm I_{2}\},\,\tfrac{-d}{c}\in [1,+\infty)\right \}\\
&\quad \cup \left \{  \left.  \tau_{>}(\tfrac{-d}{c})\right \vert
\bigl(\begin{smallmatrix}a & b\\
c & d\end{smallmatrix}\bigr)
\in SL(2,\mathbb{Z})/\{ \pm I_{2}\},\,\tfrac{-d}{c}\in (-\infty,0]\right \}\\
&\subset\mathcal{C}_<\cup\mathcal{C}_>=\overline{\mathcal{D}}\cap F=\overline{\mathcal{D}}\setminus\{\infty\}.
\end{align*}
The proof is complete.
\end{proof}

\section{Monodromy interpretation of the curves}
\label{mono-int}

The purpose of this section is to give a monodromy meaning of the curves from a complex linear ODE.
Let $\sigma(z)=\sigma(z;\tau)$ be the well-known Weierstrass sigma function defined by
$\sigma(z):=\exp \int^{z}\zeta(\xi)d\xi$.
Fix any $\tau\in F_0\setminus\{e^{\pi i/3}\}$. Then $g_2\neq 0$ and we define $\pm q_1$ and $\pm q_2$ (mod $\Lambda_{\tau}$) by
\[\wp(q_1):=\sqrt{g_2/12}, \quad \wp(q_2):=-\sqrt{g_2/12}.\]
Since $\wp''(z)=6\wp(z)^2-g_2/2$, we have $\wp''(q_1)=\wp''(q_2)=0$ and up to a constant,
\[y_1(z):=\frac{1}{\wp''(z)}=\frac{\sigma(z)^4}{\sigma(z-q_1)\sigma(z+q_1)\sigma(z-q_2)\sigma(z+q_2)}
.\]
The addition formula (cf. \cite{Akhiezer}) $\zeta(2z)-2\zeta(z)=\frac{\wp''(z)}{2\wp'(z)}$ implies
\[\zeta(\pm 2q_j)=\pm 2\zeta(q_j),\quad j=1,2,\]
and the addition formula (cf. \cite{Akhiezer})
\[\zeta(u+v)+\zeta(u-v)-2\zeta(u)=\frac{\wp'(u)}{\wp(u)-\wp(v)}\]
yields
\[\zeta(q_i+q_j)+\zeta(q_i-q_j)-2\zeta(q_i)=\frac{\wp'(q_i)}{2\wp(q_i)},\;\{i,j\}=\{1,2\}.\]
Together with $\frac{\sigma'(z)}{\sigma(z)}=\zeta(z)$,
 a direct computation leads to
\[\frac{y_1'(z)}{y_1(z)}=4\zeta(z)-\sum_{j=1}^2(\zeta(z-q_j)+\zeta(z+q_j))\]
and so
\[\frac{y_1''(z)}{y_1(z)}=\left(\frac{y_1'(z)}{y_1(z)}\right)'+\left(\frac{y_1'(z)}{y_1(z)}\right)^2=I(z;\tau),\]
where
\begin{align*}
I(z;\tau):=&12\wp(z;\tau)+2\sum_{j=1}^2(\wp(z-q_j;\tau)+\wp(z+q_j;\tau))\\
&+\sum_{j=1}^2\frac{\wp'(q_j;\tau)^2}{\wp(q_j;\tau)(\wp(z;\tau)-\wp(q_j;\tau))}
\end{align*}
is an even elliptic function. Therefore, the even elliptic function $y_1(z)$ is a solution of the following linear differential equation
\begin{equation}\label{linearODE}
y''(z)=I(z;\tau)y(z),\quad z\in\mathbb{C}.
\end{equation}

To look for another solution of (\ref{linearODE}),
we define a meromorphic function
\[\chi(z):=18\wp(z)^2\wp'(z)+\tfrac{1}{2}g_2\wp'(z)+2g_2^2z-36g_3\zeta(z),\]
which clearly has two quasi-periods:
\begin{equation}\label{chi-1}\chi_1:=\chi(z+1)-\chi(z)=2(g_2^2-18\eta_1 g_3),\end{equation}
\begin{equation}\label{chi-2}\chi_2:=\chi(z+\tau)-\chi(z)=2(g_2^2\tau-18\eta_2 g_3)=\chi_1\phi(\tau),\end{equation}
where $\phi(\tau)$ is precisely the aforementioned function in $(\ref{Cphi})$. A direction computation leads to
\[\chi'(z)=18(\wp^2\wp')'+\tfrac{1}{2}g_2\wp''(z)+2g_2^2+36g_3\wp=7(\wp'')^2=7y_1(z)^{-2}.\]
Thanks to this fact $\chi'(z)=7y_1(z)^{-2}$, it is easy to see that the meromorphic function $y_2(z):=y_1(z)\chi(z)$ is a linearly independent solution of (\ref{linearODE}) with respect to $y_1(z)$. This infers that all singularities $\{0, \pm q_1, \pm q_2\}$ (mod $\Lambda_{\tau}$) of (\ref{linearODE}) are apparent. Recall $E_{\tau}:=\mathbb{C}/\Lambda_{\tau}$. Consequently, the monodromy representation of (\ref{linearODE}) is a group homeomorphism $\rho: \pi_1(E_{\tau})\to SL(2,\mathbb{C})$, i.e. the monodromy group of (\ref{linearODE}) is generated by
\begin{equation}\label{mono}\begin{pmatrix}
1 & 0\\
1 & 1
\end{pmatrix},\quad\begin{pmatrix}
1 & 0\\
D & 1
\end{pmatrix},\quad\text{where}\; D=\phi(\tau),
\end{equation}
because it is easy to see from (\ref{chi-1})-(\ref{chi-2}) that
\begin{equation*}
\begin{pmatrix}
\chi_{1}y_{1}(z+1)\\
y_{2}(z+1)
\end{pmatrix}
=
\begin{pmatrix}
1 & 0\\
1 & 1
\end{pmatrix}%
\begin{pmatrix}
\chi_{1}y_{1}(z)\\
y_{2}(z)
\end{pmatrix}
,
\end{equation*}
\begin{equation*}
\begin{pmatrix}
\chi_{1}y_{1}(z+\tau)\\
y_{2}(z+\tau)
\end{pmatrix}
=
\begin{pmatrix}
1 & 0\\
D & 1
\end{pmatrix}%
\begin{pmatrix}
\chi_{1}y_{1}(z)\\
y_{2}(z)
\end{pmatrix},\quad D:=\frac{\chi_2}{\chi_1}=\phi(\tau).
\end{equation*}
Due to (\ref{mono}), we refer this $D$ as {\it the monodromy data} of (\ref{linearODE}). Remark that if $D=\infty$, i.e. $\chi_1=0$, then $\chi_2\neq 0$ and (\ref{mono}) should be understood as
\[\begin{pmatrix}
1 & 0\\
0 & 1
\end{pmatrix},\quad\begin{pmatrix}
1 & 0\\
1 & 1
\end{pmatrix},\]
because it is easy to see that
\begin{equation*}
\begin{pmatrix}
\chi_{2}y_{1}(z+1)\\
y_{2}(z+1)
\end{pmatrix}
=
\begin{pmatrix}
1 & 0\\
0 & 1
\end{pmatrix}%
\begin{pmatrix}
\chi_{2}y_{1}(z)\\
y_{2}(z)
\end{pmatrix}
,
\end{equation*}
\begin{equation*}
\begin{pmatrix}
\chi_{2}y_{1}(z+\tau)\\
y_{2}(z+\tau)
\end{pmatrix}
=
\begin{pmatrix}
1 & 0\\
1 & 1
\end{pmatrix}%
\begin{pmatrix}
\chi_{2}y_{1}(z)\\
y_{2}(z)
\end{pmatrix}.
\end{equation*}

Now we consider the set of those $\tau$'s in $F_0\setminus\{e^{\pi i/3}\}$ such that the monodromy data $D=\phi(\tau)$ of (\ref{linearODE}) is real-valued or $\infty$:
\[L:=\{\tau\in F_0\setminus\{e^{\pi i/3}\} \,|\, \text{the monodromy data $D\in \mathbb{R}\cup\{\infty\}$}\}.\]
Then the above argument (see particularly (\ref{c-phi})) implies that the three curves coincide with $L$.
\begin{theorem}[Monodromy interpretation] {\ }
\begin{itemize}
\item[(1)]$L=\mathcal{C}_{1}\cup\mathcal{C}_{2}\cup\mathcal{C}_{3}$ and $L\cap F=\mathcal{C}_{<}\cup \mathcal{C}_{>}$.
\item[(2)] Given $\gamma=\bigl(\begin{smallmatrix}a & b\\
c & d\end{smallmatrix}\bigr)
\in \Gamma_{0}(2)$, we let $\gamma\cdot\tilde{\tau}=\frac{a\tilde{\tau}+b}{c\tilde{\tau}+d}$ with $\tilde{\tau}\in F_0$ be a zero of $E_6'(\tau)$ in $\gamma(F_0)$. Then the monodromy data $D$ of (\ref{linearODE}) with $\tau=\tilde{\tau}$ is precisely $\frac{-d}{c}$.
\end{itemize}
\end{theorem}

\subsection*{Acknowledgements}
The research of the author was supported by NSFC (No. 12222109).

\end{document}